\documentclass[11pt, a4paper]{article}

\usepackage[colorlinks,citecolor=blue,urlcolor=blue]{hyperref}
\usepackage{amsmath}
\usepackage{amsthm}
\usepackage{amssymb}
\usepackage{amsfonts, mathrsfs}
\usepackage{cases}
\usepackage{graphicx}
\usepackage{xcolor}
\usepackage[margin=1in]{geometry}
\usepackage[utf8]{inputenc}
\usepackage{graphicx} 
\usepackage{geometry}
\usepackage{lipsum}
\usepackage{fancyhdr}
\usepackage{mathtools}
\usepackage{amsfonts}

\usepackage{verbatim}
\usepackage{caption,subcaption}
\usepackage{bm}
\usepackage{algorithm}
\usepackage{algpseudocode}
\usepackage[normalem]{ulem}

\numberwithin{equation}{section}

\newtheorem{Definition}{Definition}[section]
\newtheorem{Remark}{Remark}[section]
\newtheorem{Theorem}{Theorem}[section]
\newtheorem{Lemma}{Lemma}[section]
\newtheorem{Proposition}{Proposition}[section]
\newtheorem{Corollary}{Corollary}[section]
\newtheorem{Assumption}{Assumption}[section]

\newcommand{\be}{\begin{equation}}
  \newcommand{\ee}{\end{equation}}
\newcommand{\bee}{\begin{equation*}}
  \newcommand{\eee}{\end{equation*}}
\newcommand{\bi}{\begin{itemize}}
  \newcommand{\ei}{\end{itemize}}

\def \eps{\varepsilon}
\def \E{\mathbb{E}}

\def \N{\mathbb{N}}
\def \P{\mathbb{P}}

\def \R{\mathbb{R}}

\def \Pc{{\mathcal P}}

\def \Bc{{\mathcal B}}
\def \Sc{{\mathcal S}}
\def \Hc{{\mathcal H}}
\def \Kc{{\mathcal K}}
\def \Xc{{\mathcal X}}
\def \Dc{{\mathcal D}}
\def \wt{\operatorname{{\bf W}}}

\def \argmax{\operatorname{argmax}}
\def \Leb{\operatorname{Leb}}
\def \supp{\operatorname{\texttt{supp}}}
\title{Time-Inconsistent MDPs with Entropy Regularization: Equilibrium Existence and Policy Iteration}
\author{
Fengyuan Cao\thanks{Zhongtai Securities Institute for Financial Studies, Shandong University, Jinan, Shandong, China. Email: \url{202612022@mail.sdu.edu.cn}}
\and Zhenhua Wang\thanks{Zhongtai Securities Institute for Financial Studies, Shandong University, and Shandong Province Key Laboratory of Financial Risk, Jinan, Shandong, China. Email:  \url{zhenhuaw@sdu.edu.cn}}
}
\date{}

\begin{document}

\maketitle

\begin{abstract}
We study infinite-horizon time-inconsistent Markov decision processes with a countably infinite state space and unbounded reward functions. The reward is allowed to depend explicitly on the initial time and initial	state, thereby accommodating general sources of time inconsistency. We seek relaxed feedback equilibria, and our approach is based on entropy regularization and weighted functional analytic methods. With entropy regularization, we characterize a regular relaxed equilibrium through a fixed-point operator. By introducing two weight functions with distinct roles---one controlling the growth of rewards and values and the other defining the ambient weighted space---we construct a compact invariant set under a product topology and apply the Schauder--Tychonoff fixed-point theorem to establish existence of regularized equilibria.	Importantly, the invariant set can be chosen uniformly for small entropy weight $\lambda\in(0,1]$. We then let $\lambda\to0+$ and show, through compactness, concentration of Gibbs policies, and uniform-integrability arguments, that a subsequential limit is a relaxed equilibrium of the	original unregularized problem.
	
We further study a policy iteration algorithm (PIA) for the entropy-regularized equilibrium problem. Under a weighted-discounting structure and sufficiently strong discounting, we establish exponential convergence and uniqueness of the regularized equilibrium in a suitable weighted Banach space. Combining the policy-iteration error with a quantitative soft-max approximation bound, we show that the iterated policies constitute weighted $\varepsilon$-equilibria for the original unregularized problem and derive an explicit regret estimate. A numerical example illustrating the convergence of PIA under strong discounting and a counterexample demonstrating its failure under weak discounting are also provided.\\

\end{abstract}

\section{Introduction}\label{sec:introduction}

Time inconsistency arises when a policy that is preferred by a decision maker at the current time ceases to be preferred when the same decision problem is reconsidered at a later time or state. Classical examples include non-exponential discounting, mean--variance objectives, state-dependent risk preferences, and objective functionals depending explicitly on the initial time or initial state. In such problems the Bellman principle generally fails, and a globally optimal dynamically consistent policy need not exist. Following the idea of Strotz \cite{Strotz55}, a now standard approach is to regard the decision problem as an intra-personal game among the agent's successive selves and to seek a subgame-perfect equilibrium policy, namely a policy under which no current self benefits from a local deviation when future selves are assumed to follow the prescribed strategy. This game-theoretic approach has generated an extensive literature on time-inconsistent stochastic control and optimal stopping; see, among others, \cite{bayraktar2023equilibria, BKM17, bjork2021time, bjork2014mean, EL06, HZ21, yong2024optimal} and the references therein. 

The corresponding theory for discrete-time Markov decision processes (MDPs) has developed along several directions. A substantial part of the literature studies equilibria in pure Markov strategies. For quasi-hyperbolic discounting, \cite{balbus2020markov,balbus2018uniqueness,JN21} establish Markov-perfect or time-consistent equilibrium results under various state-space and structural assumptions; in the countable-state setting, \cite{mei2021time} considers a time-inconsistent risk-sensitive criterion. More general discrete-time formulations and existence questions for Markov equilibrium controls are studied in \cite{balbus2022time,bayraktar2023existence}. These papers cover finite, countable, and more general state spaces and criteria ranging from quasi-hyperbolic discounting to recursive or risk-sensitive preferences. Most of these results focus on pure Markov equilibria; \cite{JN21} also studies randomized stationary equilibria on general Borel state spaces and obtains deterministic equilibria under additional conditions, including results for countable-state models. The restriction to pure strategies is nontrivial: as demonstrated in \cite{bayraktar2025relaxed}, a pure equilibrium may fail to exist even for a finite-state time-inconsistent MDP, which motivates the use of relaxed policies in more general models.

Relaxed strategies have therefore become increasingly important in the existence theory. \cite{bayraktar2025relaxed} study infinite-horizon time-inconsistent MDPs with general non-exponential discounting on a finite state space. They introduce entropy regularization, establish an equilibrium for the regularized problem, and then let the entropy weight vanish to obtain a relaxed equilibrium for the original problem. The present paper builds on this idea in a broader setting with a countably infinite state space, possibly unbounded rewards, time-inhomogeneous transition probabilities, and explicit initial-time/state dependence. These extensions require weighted compactness and uniform-integrability arguments to handle the infinite-dimensional state space and unbounded value functions.

Related vanishing-entropy methods have also been developed for continuous-time time-inconsistent control \cite{wang2026equilibrium} and for time-inconsistent mean field games \cite{BayraktarWangYuZhang2026}. These developments suggest that entropy regularization provides a common framework for equilibrium existence, approximation, and computation under time inconsistency.

In the context of reinforcement learning, \cite{AB10, Fedus20, SRK22} develop algorithms for MDPs with non-exponential or hyperbolic discounting. These works primarily aim at computing optimal, or precommitment, policies. Such a policy is chosen from the perspective of the initial self and, under time inconsistency, need not remain optimal when the decision problem is reconsidered later. Our focus is instead on equilibrium policies for a sophisticated agent, for which sequential incentives are built directly into the solution concept.

It is also worth mention two recent works. \cite{jaskiewicz2026stochastic} establish randomized Markov-perfect equilibria for stochastic games of risk-sensitive players with quasi-hyperbolic discounting on a countable state space, allowing time-inhomogeneous primitives but assuming uniformly bounded one-stage payoffs. \cite{zhou2026existence} develops an existence theory for Markov relaxed equilibria in discrete-time time-inconsistent stochastic games, including multi-player, Stackelberg, and mean-field settings, with particular emphasis on uncountable state spaces and weak-star compactness. In contrast to these works, our framework allows unbounded rewards, while entropy regularization plays a central role not only in equilibrium existence but also in numerical computation, including convergence of the associated policy iteration algorithm and quantitative weighted $\eps$-equilibrium guarantees for the original unregularized problem.

The present paper develops an entropy-regularization framework for equilibrium existence and computation in infinite-horizon time-inconsistent MDPs. We work with a countably infinite state space, time-inhomogeneous transition probabilities, and possibly unbounded rewards that may depend on both the evaluating time-state pair and the future time-state pair. This formulation covers general sources of time inconsistency beyond non-exponential discounting.

To handle the resulting unboundedness, we employ weighted functional-analytic techniques from the theory of Markov control processes; see, e.g., \cite{hernandez2012further}. A distinctive feature of our construction is the use of two weights $w$ and $W$: the former controls the growth of rewards and value functions, while the latter defines the ambient weighted Banach space $\Bc_{\wt}$. The continuation values are then placed in the product space $\Xc$. The strict separation between the two weights turns uniform $w$-growth bounds into uniformly vanishing $W$-tails, yielding compactness of the relevant invariant set in the product topology (see Lemma~\ref{lm:KM.compact}). This allows us to accommodate unbounded rewards and infinitely many calendar-time coordinates without imposing artificial decay as $t\to\infty$.

Our first main result establishes the existence of an entropy-regularized equilibrium; see Theorem~\ref{thm:existence.regular}. The main difficulties arise from the infinite state space, unbounded rewards, and time-inhomogeneous dynamics. Weighted entropy and moment estimates (Lemmas \ref{lm:entropy.est} and \ref{lm:est.0}) allow us to construct a compact convex invariant set for the regularized value operator $\Psi_\lambda$ (see Lemma \ref{lm:self-mapping}). Weighted tail estimates, together with pointwise convergence of the Gibbs policies and finite-step transition laws, yield continuity of $\Psi_\lambda$ in Lemma \ref{lm:Psi.continuous}. The Schauder--Tychonoff fixed-point theorem then gives a regular relaxed equilibrium. Importantly, the invariant growth bound is uniform over small entropy weights, providing the compactness needed for the subsequent vanishing-entropy analysis.

Our second main result establishes the existence of a relaxed equilibrium for the original problem through a vanishing-entropy argument (see Theorem~\ref{thm:relax.origin}). Starting from regularized equilibria along a vanishing sequence of entropy weights, the uniform estimates from Section~\ref{sec:entropy} yield subsequential compactness of both the auxiliary values and policies. Then Lemma~\ref{lm:zero.entropylimit} uses the concentration property of the Gibbs policies to identify the support of the limiting policy, while Lemma~\ref{lm:zero.entropylimitV} establishes convergence of the unregularized values through weighted uniform-integrability estimates. Together with Proposition~\ref{prop:fixedpoint.origin}, these results identify the limiting value with that generated by the limiting policy and yield a relaxed equilibrium of the original problem.

Entropy regularization has also become an important analytical and algorithmic tool in reinforcement learning and stochastic control. In the time-consistent diffusion setting, convergence of entropy-regularized policy iteration has recently been studied by \cite{HuangWangZhou2025convergence} using analytic estimates and by \cite{ma2026convergence} using probabilistic representations. Under time inconsistency, the usual policy-improvement principle is no longer available. Recently, \cite{huang2026policy} addressed this difficulty for finite-horizon continuous-time diffusion control and established exponential convergence of an entropy-regularized policy iteration algorithm.

Our third contribution develops an infinite-horizon, countable-state counterpart of this algorithmic program. We formulate a policy iteration algorithm (PIA) under the weighted-discounting structure and derive the key stability and uniform estimates in Lemmas~\ref{lm:PIA.est1} and~\ref{lm:PIA.rho.uniform}. These estimates yield a contraction in the weighted Banach space $\Xc$. Consequently, Theorem~\ref{thm:converge.PI} establishes exponential convergence of PIA under sufficiently strong discounting, while Corollary~\ref{cor:PIA.unique} gives uniqueness of the regularized equilibrium in $\Xc$. In contrast, the explicit two-state example in Section~\ref{subsec:PIA.nonconvergent.interval} shows that, under weak discounting, the regularized equilibrium may be unique while PIA is instead attracted to a nontrivial period-two orbit.

More importantly, Corollary \ref{cor:PIA.epsilon} provides a quantitative connection between the PIA iterates and equilibria of the original zero-entropy problem. Combining a uniform soft-max approximation estimate with the exponential convergence in Theorem \ref{thm:converge.PI}, we prove in Corollary \ref{cor:PIA.epsilon} that the $n$th PIA policy is a weighted $\eps_n$-equilibrium, with
$$
\eps_n \leq C\left[ \lambda(1+|\ln\lambda|) + \left( \frac{\rho_0 C_1\lambda^{-1}} {1-\rho_0 C_0} \right)^{n-1} \|v^0-v^*\|_{\Xc} \right].
$$
The above regret bound is determined jointly by regularization bias and the PIA error, and the two effects therefore have to be balanced when using PIA to approximate an equilibrium of the original problem. To the best of our knowledge, this is the first quantitative regret analysis for PIA in time-inconsistent control and the first result that converts convergence of a computational algorithm into an explicit $\eps$-equilibrium guarantee for the underlying time-inconsistent problem.

The rest of the paper is organized as follows. Section \ref{sec:model} introduces the model and equilibrium concepts, develops the weighted functional framework, and states the main assumptions.
Section \ref{sec:entropy} studies the entropy-regularized problem, derives the Gibbs fixed-point formulation, and establishes existence of regular relaxed equilibria. Section \ref{sec:origin} studies the vanishing-entropy limit and proves existence of relaxed equilibria for the original problem. Section \ref{sec:pia} analyzes policy iteration under weighted discounting, establishes exponential convergence and uniqueness under strong discounting, derives the quantitative equilibrium-regret estimate, and presents a numerical convergence example under strong discounting and a nonconvergence example under weak discounting.

\section{Models and preliminaries}\label{sec:model}
Set $\N_0 \coloneqq \N \cup \{0\}$. Let $A \subset \R^l$ be a compact action space with $\Leb(A)>0$. Consider a probability space $(\Omega, \mathcal{F}, {P})$ that supports a  discrete-time Markov process $(X_t)_{t\in \N_0}$ that takes values in $\Sc=\{1,2,3,\dots\}$. The dynamics of $X$ is governed by the following transition function
$$
p(t,i,j,a):=\P(X_{t+1}=j | X_t=i , \alpha_t=a)\quad  \forall t\in \N_0, \; i,j \in \Sc, \; a\in A.
$$
Given $t\in \N_0$ and $a\in A$, we will denote by $p(t,i,a)$ the vector $(p(t,i,j,a))_{j\in \Sc}$ and $p(t,a)$ the matrix $(p(t,i,j,a))_{i,j\in \Sc}
$.

Let $\Pc(A)$ (resp. $\Dc(A))$  denote the set of all probability measures (resp. density functions) on the action space $A$. We call $\pi  :\N_0\times \Sc\rightarrow \Pc(A) $ a {\it relaxed  feedback control}, and we denote by $\Pi$  the set of all relaxed feedback controls. For arbitrary $t\in \N_0$ and $i\in \Sc$, we also denote by $p^\varpi(t,i)$ the vector of probabilities $(\int_A p(t,i,j,a)\varpi(da))_{j\in \Sc}$ under a given $\varpi\in \Pc(A)$, and denote by $p^\pi(t) = (\int_A p(t,i,j)\pi(t,i,da))_{i,j\in \Sc}$ the transition matrix  for a given $\pi\in \Pi$.

Given $\pi \in \Pi$, the dynamics of $X = X^\pi$ is determined as follows. At any time $t \in \mathbb{N}_0$, given that  
$X_t = i \in \mathcal{S}$, we sample $a \in A$ according to the probability measure $\pi(t,i) \in \mathcal{P}(A)$. The realization  
of $X_{t+1}$ is then governed by the transition probabilities $p(t,i,a) = (p(t,i,j,a))_{j \in \mathcal{S}}$.

Define the subset $\Delta:=\{(t,s)\in \N_0\times \N_0, 0\leq t\leq s\}.$ Given a reward function $f(s, j, a; t, i)$ with $(t,s)$, we define $  
f^\varpi(s,j;t,i) := \int_A f(s, j, a; t, i) \varpi(da) $
for all $\varpi \in \mathcal{P}(A)$. For any $\pi \in \Pi$, the corresponding infinite-horizon pay-off functional is given by
\begin{equation}\label{1.2}
J^\pi(t,i) := \mathbb{E}_{t,i} \left[ \sum_{k=t}^\infty f^{\pi(k, X^\pi_{k})}(k, X_{k}^\pi;t,i) \right], \quad \forall (t,i)\in \N_0\times\mathcal{S}.   
\end{equation}
Due to the dependence structure of $f$ on the initial time and spatial states $t,i$, such a control problem is time-inconsistent. A special case is non-exponential discounting, e.g., $f(s,y;t,x)=\delta(s-t)g(y)$. 
\begin{equation}\label{1.3}
f(s,j, a;t,i)=\widehat \delta(s-t)g(j,a)
\end{equation}
where $g : \Sc\times A \to \R$ assigns a reward based on the current state $j$ and the action $a$ employed, and $\widehat \delta : [0, \infty) \to [0, 1]$ is a discount function, assumed to be nonincreasing with $\widehat \delta(0) = 1$.

In this paper, we will investigate an entropy-regularized version of $J^\pi(t, i)$. To this end, let us first introduce the notion of a regular relaxed feedback control.
A relaxed feedback control $\pi \in \Pi$ is regular if, for each $(t,i) \in \N_0\times S$, $\pi(t,i)$ is a density function belonging to $\Dc(A)$ such that $\mathcal{H}(\pi(t,i)) > -\infty$, where
\begin{equation}\label{1.4}
\mathcal{H}(\varpi) := - \int_A \ln(\varpi(a)) \varpi(a) da, \quad \forall \varpi \in \Dc(A).  
\end{equation}
We denote by $\Pi_r$ the subset of $\Pi$ containing all regular feedback relaxed controls. 

Now, for any $\lambda > 0$ and $\pi \in \Pi_r$, consider
\begin{equation}\label{1.5}
J_\lambda^\pi(t, i) := \mathbb{E}_{t,i} \left[ \sum_{k=t}^\infty \left( f^{\pi(k,X_k^\pi)}(k, X_k^\pi; t, i) + \lambda \widetilde\delta(k-t) \mathcal{H}\left(\pi(X_k^\pi)\right) \right) \right], \quad \forall i \in S,
\end{equation}
where $\widetilde\delta : \N_0 \to [0, 1]$ is an \textit{arbitrarily} chosen discount function which is nonincreasing with $\widetilde\delta(0) = 1$ and satisfies $\sum_{k=0}^\infty\widetilde\delta(k)<\infty$. 

\begin{Definition}\label{def:equilibrium}
Given $\lambda > 0$, we say $\pi \in \Pi_r$ is a regularized equilibrium if for any $(t,i)\in \N_0\times \Sc$,
\be \label{D1}
J_\lambda^{\varpi \otimes_1 \pi}(t, i) - J_\lambda^\pi(t, i) \leq 0, \quad \forall \varpi\in \Pc(A).  
\ee 
Similarly, we say $\pi \in \Pi$ is a relaxed equilibrium for the original problem if for any $(t,i) \in \N_0\times \Sc$,
\be \label{D2}
J^{\varpi \otimes_1 \pi}(t, i) - J^\pi(t, i) \leq 0, \quad \forall \varpi \in \Pc(A).    
\ee 
\end{Definition}

\subsection{Preliminaries and assumptions}

We introduce two weight functions $w, W: \Sc \to [1,\infty)$ satisfying
\be\label{eq: assum.wW}
\begin{cases}
\lim_{i\to \infty} w(i) = \lim_{i\to \infty} W(i)=\infty, \quad \lim_{i\to\infty} \frac{w(i)}{W(i)} = 0;\\
C_{W}:=\sup_{i\in \Sc}\frac{\ln(W(i))}{w(i)}<\infty.
\end{cases}
\ee
The first condition in \eqref{eq: assum.wW} implies that 
\be\label{eq: assum.wW0}  
C_w:=\sup_{i\in \Sc} \frac{w(i)}{W(i)}<\infty
\ee 
For a function $y:\Sc\to\mathbb R$, define the weighted norm 
$$\|y\|_{\wt}:= \sup_{i\in \Sc}\frac{|y(i)|}{W(i)},$$
and let
$
\Bc_{\wt}:=\{ y:\Sc\to \R: \|y\|_{\wt}<\infty \}.
$
We further define the space $\Xc:= \prod_{(t,i)\in \N_0\times \Sc} \Bc_{\wt}$, whose elements are functions $v(j;t,i): \Sc\times \N_0\times \Sc \to\R$  such that 
$$
\|v(\cdot;t,i) \|_{\wt}<\infty,\quad \forall (t,i)\in \N_0\times \Sc.
$$
We equip $\Xc$ with the product topology. Equivalently, a sequence
$\{v^n\}\subset\Xc$ converges to $v\in\Xc$ if and only if
$$
\|v^n(\cdot;t,i)-v(\cdot;t,i)\|_{\wt}\to 0, \quad\forall\,(t,i)\in\N_0\times\Sc.
$$
For convenience, we will also commonly treat $v(\cdot; t,i)$ as a vector in $\R^\infty,$ i.e., $$
v(\cdot; t,i) = (v(1;t, i), v(2; t, i), \dots).$$ 

Given a constant $M\in (0,\infty)$, we further introduce the subset $\Kc(M)$ of $\Xc$ as 
\be\label{eq:def.KM}  
\Kc(M):= \{ v(j;t,i)\in \Xc: |v(j;t,i)|\leq M(w(i)+w(j))\quad \forall (j,t,i)\in  \Sc\times \N_0\times \Sc\}.
\ee 
$\Kc(M)$ contains functions with growth on state variables $i, j$ bounded by the weight function $w$. Moreover, $\Kc(M)$ forms a compact subset of $\Xc$ as stated below.
\begin{Lemma}\label{lm:KM.compact}
For any $M\in (0,\infty)$, $\Kc(M)$ is a compact and convex subset of $\Xc$ under the product topology.
\end{Lemma}

\begin{proof}
 For every $(t,i)\in\N_0\times\Sc$, define
  $$
 K_{t,i}(M) := \left\{ y\in\Bc_W: |y(j)| \leq M\bigl(w(i)+w(j)\bigr), \quad \forall j\in\Sc \right\}.
  $$
 Then, by the definition of $\Kc(M)$ in \eqref{eq:def.KM},
  $
 \Kc(M) = \prod_{(t,i)\in\N_0\times\Sc}K_{t,i}(M).
  $ 
 We first show that $K_{t,i}(M)$ is compact in $\Bc_W$ for every fixed
 $(t,i)$. Clearly, $K_{t,i}(M)$ is convex. Moreover, for any
 $y\in K_{t,i}(M)$,
  $$
 \frac{|y(j)|}{W(j)} \leq M\left( \frac{w(i)}{W(j)} + \frac{w(j)}{W(j)} \right).
  $$
 Since $W(j)\to\infty$ and $w(j)/W(j)\to0$ by
 \eqref{eq: assum.wW}, it follows that
 \be
 \label{eq:KM.uniform.tail} \lim_{R\to\infty} \sup_{y\in K_{t,i}(M)} \sup_{j>R} \frac{|y(j)|}{W(j)} =0.
 \ee
 In addition, by \eqref{eq: assum.wW0},
  $$
 \|y\|_{\wt} \leq M\left( w(i)+C_w \right), \quad \forall y\in K_{t,i}(M),
  $$
 so $K_{t,i}(M)$ is bounded in $\Bc_W$.
 
 We now prove total boundedness. Fix $\eps>0$. By
 \eqref{eq:KM.uniform.tail}, there exists $R\in\N$ such that
  $$
 \sup_{y\in K_{t,i}(M)} \sup_{j>R} \frac{|y(j)|}{W(j)} <\frac{\eps}{2}.
  $$
 On the finite-dimensional coordinates $\{1,\ldots,R\}$, 
  $$
\text{the set } \prod_{j=1}^{R} \left[ -M\bigl(w(i)+w(j)\bigr), M\bigl(w(i)+w(j)\bigr) \right] \text{  is compact.}
  $$
Hence it admits a finite $\eps/2$-net under the norm
  $
 \max_{1\leq j\leq R}({|y(j)|}/{W(j)}).
  $
By extending each element of this finite net by zero on $\{j>R\}$, we obtain a finite $\eps$-net of $K_{t,i}(M)$ in $\Bc_W$. Therefore $K_{t,i}(M)$ is totally bounded.
 
 We next show that $K_{t,i}(M)$ is closed in $\Bc_W$. Let
 $y^n\in K_{t,i}(M)$ and suppose that $\|y^n-y\|_{\wt}\to0$. Then, for
 each fixed $j\in\Sc$,
  $$
 |y^n(j)-y(j)| \leq W(j)\|y^n-y\|_{\wt} \longrightarrow0.
  $$
 Since, for any $n\in \N$, $ |y^n(j)| \leq M\bigl(w(i)+w(j)\bigr),$ passing to the limit gives
  $
 |y(j)| \leq M\bigl(w(i)+w(j)\bigr),
  $
 so $y\in K_{t,i}(M)$. Hence $K_{t,i}(M)$ is closed. Since $\Bc_W$ is a Banach space, total boundedness and closedness imply that $K_{t,i}(M)$ is compact.
 
 Finally, each $K_{t,i}(M)$ is compact and convex. Therefore, by
 Tychonoff's theorem,
  $$
 \Kc(M) = \prod_{(t,i)\in\N_0\times\Sc}K_{t,i}(M) \text{   is compact under the product topology of $\Xc$. }
  $$
Its convexity follows immediately from the convexity of all the coordinate sets $K_{t,i}(M)$.
\end{proof}

\begin{Remark}
One simple choice is to take $w$ to be a polynomial growth function and set $W=w^{1+\eps}$. 
  Also, the product topology on $\Xc$ is generated by the metric
  $$d(v, v'):= \sum_{(t,i)\in \N_0\times \Sc}\frac{1}{2^{t+i}} \left(  \|v(\cdot; t, i)- v'(\cdot; t, i)\|_{\wt} \wedge 1\right).$$
\end{Remark}

Now we introduce the assumptions on the reward function $f$ and transition function $p$.
\begin{Assumption}[Weighted boundedness of the reward function]\label{assum:reward-bound}
  There exists a constant $C_f \in(0,\infty)$ and a discount function $\delta(t):\N_0\to [0,\infty)$ such that for any $(t,s,i,j, a)\in\Delta\times\Sc^2\times A$:
  \begin{equation}\label{eq:assume.fgrow}
 |f(s,j, a; t, i)| \leq C_f \delta(s-t) [w(i)+w(j)].  
  \end{equation}
  And the discount function satisfies 
  $
  \sum_{t=0}^\infty \delta(t) < \infty.
  $
\end{Assumption}

\begin{Assumption}[Weighted boundedness of the transition function]\label{assum:lyapunov.p}
There exist constants $\gamma_w, \gamma_{W} \in (0,1)$ and $b_w, b_{W}\in (0,\infty)$ such that the following hold uniformly over $(t,i,a)$:
  \begin{align}
&\sum_{j \in S} p(t,i,j,a) w(j) \leq \gamma_w w(i) + b_w,   \label{eq:assume.pw}\\
&\sum_{j \in S} p(t,i,j,a) W(j) \leq \gamma_{W} W(i) + b_{W}. \label{eq:assume.pW} 
  \end{align}
\end{Assumption}

\begin{Remark}\label{rm:payoff.welldefined}
Assumptions \ref{assum:reward-bound}--\ref{assum:lyapunov.p} ensure that the pay-off functional in \eqref{1.2} is well-defined for every $\pi\in\Pi$. Indeed, iterating \eqref{eq:assume.pw} gives
$$
\E\left[w(X_s^\pi)\mid X_t^\pi=i\right]	\leq\gamma_w^{\,s-t}w(i)+\frac{b_w}{1-\gamma_w}, \quad s\geq t.
$$
Therefore, by Assumption \ref{assum:reward-bound}
$$
\begin{aligned}
\E\left[\left| f^{\pi(s,X_s^\pi)} (s,X_s^\pi;t,i) \right| \,\middle|\, X_t^\pi=i \right] &\leq C_f\delta(s-t)\left[ w(i)+ \E\left[ w(X_s^\pi)\mid X_t^\pi=i \right] \right] \\
&\leq C_f\delta(s-t) \left[ (1+\gamma_w^{\,s-t})w(i) + \frac{b_w}{1-\gamma_w} \right].
\end{aligned}
$$
As a result,
$$
\begin{aligned}
\E\left[ \sum_{s=t}^{\infty} \left| f^{\pi(s,X_s^\pi)} (s,X_s^\pi;t,i) \right| \right] &\leq C_f \sum_{m=0}^{\infty} \delta(m) \left[ (1+\gamma_w^m)w(i) + \frac{b_w}{1-\gamma_w} \right] <\infty,
\end{aligned}
$$
where the last inequality follows from $\sum_{m=0}^{\infty}\delta(m)<\infty$ and $\gamma_w\in(0,1)$. Thus the series defining $J^\pi(t,i)$ converges absolutely in $L^1$, and in particular $J^\pi(t,i)$ is finite for every $(t,i)\in\N_0\times\Sc$.
	
For the entropy-regularized pay-off $J_\lambda^\pi$ defined in \eqref{1.5}, an additional integrability condition on the entropy term is needed. For the Gibbs policies considered in Section \ref{sec:entropy}, such an estimate is provided by Lemma \ref{lm:entropy.est} together with the above moment estimate, and therefore the corresponding entropy-regularized pay-off is also well-defined.
\end{Remark}

\begin{Assumption}[Weighted Lipschitz condition on control]\label{assum:lipschitz}
There exist constants $L_f, L_p\in (0,\infty)$ such that the following hold uniformly for any $t, s, i, j$ and any pair of actions $a_1, a_2 \in A$:
\begin{align}
&|f(s, j, a_1; t, i) - f(s, j, a_2; t, i)| \leq L_f w(i) |a_1 - a_2|, \label{eq:assume.lipaf}\\
& \sum_{j \in S} |p(t,i,j, a_1) - p(t,i,j, a_2)|{W}(j) \leq L_p W(i) |a_1 - a_2|.  \label{eq:assume.lipap}
\end{align}
\end{Assumption}

We will also assume that the action space $A \subset \R^\ell$ fulfills a \textit{uniform cone condition}. To properly state the condition, for any $\iota \in [0, \pi/2]$, we note that
\bee
\textsf{cone}_\iota := \{a = (a_1, \dots, a_\ell) \in \R^\ell : a_1^2 + \dots + a_{\ell-1}^2 \leq \tan^2(\iota) a_\ell^2\}  
\eee
is a cone in $\R^{\ell}$ with vertex, axis, and angle being $0$, $a_1 = a_2 = \dots = a_{\ell-1} = 0$, and $\iota$, respectively. Now, given $a \in \R^\ell$, we denote by $\textsf{cone}_\iota(a)$ a rotated copy of $a + \textsf{cone}_\iota$ about $a$ in $\R^\ell$, which will be called \textit{a cone with vertex $a$ and angle $\iota$}.

\begin{Assumption}[Uniform cone condition]\label{assum:cone}
When $\ell > 1$, there exist $\vartheta > 0$ and $\iota \in (0, \pi/2]$ such that for any $a \in A$, there is a cone $\textsf{cone}_\iota(a)$ that satisfies $(\textsf{cone}_\iota(a) \cap B_\vartheta(a)) \subseteq A$. When $\ell = 1$, there exists $\vartheta > 0$ such that for any $a \in A$, either $[a - \vartheta, a]$ or $[a, a + \vartheta]$ is contained in $A$.
\end{Assumption}

\section{Existence of regularized equilibria for $\lambda>0$}\label{sec:entropy}
Let us first focus on the existence of regularized equilibria. Fix $\lambda > 0$. For any $\pi \in \Pi_r$, let us introduce the auxiliary value function $V^\pi_\lambda: \Sc\times \N_0\times\Sc\to \R$ as 
\begin{equation}\label{eq:def.V}
  V_{\lambda}^{\pi}(j;t, i) := \mathbb{E}\left[ \sum_{k=t+1}^{\infty} \left( f^{\pi(k, X_k^{\pi})}(k, X_k^{\pi}; t, i) + \lambda \widetilde\delta(k-t) \Hc(\pi(k, X_k^{\pi})) \right) \bigg| X^\pi_{t+1}=j\right].
\end{equation}
Then we observe from \eqref{1.5} and \eqref{eq:def.V} that
$$
\begin{aligned}
  J_{\lambda}^{\varpi \otimes_1 \pi}(t,i) =& f^{\varpi}(t,i;t,i) + \lambda \mathcal{H}(\varpi) + \mathbb{E}_{t,i} \left[ V_{\lambda}^{\pi}(X_{t+1}^{\varpi}; t, i) \right]\\
  = &\int_A \left[ f(t, i, a;t,i) - \lambda \ln((\varpi)(a)) + p(t,i,a) \cdot V_{\lambda}^{\pi}(\cdot; t, i) \right] \varpi(a) da.
\end{aligned}
$$
We set the Hamiltonian $H: \N_0\times \Sc\times\Bc_{\wt}\times A \to \R$ as
\be\label{eq:def.H}
H(t,i,y,a) :=  f(t,i,a;t,i) + p(t,i,\cdot,a)\cdot y.
\ee
Notice that the Hamiltonian in \eqref{eq:def.H} is
well-defined for every $y\in\Bc_{\wt}$. Indeed, for any
$(t,i)\in\N_0\times\Sc$ and $a\in A$,
$$
\begin{aligned}
\sum_{j\in\Sc}p(t,i,j,a)|y(j)|&\leq\|y\|_{\wt}\sum_{j\in\Sc}p(t,i,j,a)W(j)\leq\bigl(\gamma_WW(i)+b_W\bigr)\|y\|_{\wt}<\infty,
\end{aligned}
$$
where the second inequality follows from \eqref{eq:assume.pW}. So does $\Gamma_\lambda(t,y,i)$ defined below in \eqref{2.5}.

Hence, the best strategy $\pi^* \in \Dc(A)$ for the agent at the joint state $(t,i)$ should satisfy
\begin{equation}\label{2.3}
  \pi^*\in \arg\max_{\varpi \in \Dc(A)} \int_A \bigg[ H\left(t,i,  V_\lambda^\pi(\cdot; t,i), a\right)- \lambda \ln(\varpi(a)) \bigg] \varpi(a) da. 
\end{equation}
The maximizer on the right-hand side is unique and takes the explicit form:
\begin{equation}\label{2.4}
  \pi^*(t,i,a) := \frac{e^{\frac{1}{\lambda} \left[ f(t, i, a;t,i) +p(t,i,a) \cdot V_\lambda^\pi(\cdot; t, i)  \right]}}{\int_A e^{\frac{1}{\lambda}  \left[ f(t, i, a';t,i) +p(t,i,a') \cdot V_\lambda^\pi(\cdot; t, i)  \right]} da'}.
\end{equation}
As a result, we introduce a functional $\Gamma_\lambda : \N_0\times \Bc_{\wt}\times\mathcal{S} \to \mathcal{D}(A)$ defined by
\begin{equation}\label{2.5}
  \Gamma_\lambda(t, y, i)(a) := \frac{e^{\frac{1}{\lambda}(f(t,i,a;t,i) + p(t,i,a) \cdot y)}}{\int_A e^{\frac{1}{\lambda}(f(t,i,a' ;t,i) + p(t, i, a') \cdot y)} da'} \in \mathcal{D}(A),
\end{equation}
and we write the mapping $\Gamma_\lambda:\Xc \mapsto \Pi_r$ as
$$
\Gamma_\lambda(v)(t,i)=\Gamma_\lambda(t, v(\cdot; t, i), i)\quad \forall (t,i)\in \N_0\times \Sc.
$$
Then we can rewrite \eqref{2.4} as
\begin{equation}\label{2.6}
  \pi^*(t,i) = \Gamma_\lambda(V_\lambda^\pi)(t,i)\quad \forall (t,i)\in \N_0\times \Sc.
\end{equation}
This motivates us to define operators $\Phi_\lambda: \Pi_r \to \Pi_r$ 
by
\begin{align}
  \Phi_\lambda(\pi)(t,i) :=& \Gamma_\lambda(V^\pi_\lambda)(t,i)= \Gamma_\lambda(t, V_\lambda^\pi (\cdot; t, i),  i),\quad \forall (t,i)\in \N_0\times \Sc  \label{2.7}
\end{align}
and we conjecture that a fixed point of $\Phi_\lambda$ is a regularized equilibrium for \eqref{1.5}. In addition, 
we define an operator $\Psi_\lambda : \Xc \to \Xc$ by
\begin{equation}\label{2.8}
  \Psi_\lambda(v) := V_\lambda^{\Gamma_\lambda(v)}
\end{equation}
and conjecture that a fixed point of $\Psi_\lambda$ must equal $V_\lambda^\pi$ for some regularized equilibrium $\pi \in \Pi_r$. Proposition \ref{prop:fixed-points-entropy} below shows that our conjectures are correct, which leads to the existence of regularized equilibrium once we verify the existence of fixed points of $\Psi_\lambda$.

\begin{Proposition}[Regularized equilibria as fixed points]\label{prop:fixed-points-entropy}
  For a given $\lambda>0$, suppose $\pi\in\Pi_r$ satisfies that $V^\pi_\lambda\in \Xc$, then $\pi$ is a regularized relaxed equilibrium if and only if
  $
  \Phi_\lambda(\pi)=\pi.
  $
  Moreover, for any $v\in \Xc$ such that $\Psi_\lambda(v)\in \Xc$,
  $$
  v=V_\lambda^\pi \text{ for some regularized equilibrium }\pi\in\Pi_r
  \quad\Longleftrightarrow\quad
  \Psi_\lambda(v)=v.
  $$
  In particular, if $\Psi_\lambda(v)=v\in \Xc$, then $\Gamma_\lambda(v)\in\Pi_r$ is a regularized equilibrium.
\end{Proposition}

\begin{proof}
  Fix an arbitrary $\lambda > 0$ and take $(t,i)\in \N_0\times \Sc$. Given $\varpi\in \Dc(A)$ and $\pi \in \Pi_r$ with $V^\pi_\lambda\in \Xc$, since $J_{\lambda}^{\varpi \otimes_1 \pi}(t, i) = f^{\varpi}(t, i;t,i) + \lambda \Hc(\varpi) + \mathbb{E}_{t,i}[V_{\lambda}^{\pi}(X_1^{\varpi}; t, i)]$, a direct calculation shows
  \begin{align*}\label{a1}
 J_{\lambda}^{\varpi \otimes_1 \pi}(t, i) - J_{\lambda}^{\pi}(t, i) 
 &= f^{\varpi}(t,i;t,i) + \lambda \Hc(\varpi) + \int_A (p(t,i,a) \cdot V_{\lambda}^{\pi}(\cdot, t, i))\varpi(a) da \\
 &\quad - \left( f^{\pi(t,i)}(t,i;t,i) + \lambda \Hc(\pi(t,i)) + \int_A (p(t,i,a) \cdot V_{\lambda}^{\pi}(\cdot; t, i))\pi(t,i)(a) da \right).
  \end{align*}
  It then follows that \eqref{D1} holds 
   if and only if $\pi(t, i) \in \mathcal{D}(A)$ fulfills
  \begin{equation}\label{3.1}
 \pi(t, i) \in \argmax_{\varpi \in \mathcal{D}(A)}\left\{ \int_A \left[ f(t,i, a; t, i) - \lambda \ln(\varpi(a)) + p(t,i, a) \cdot V_{\lambda}^{\pi}(\cdot; t, i) \right] \varpi(a)da \right\}. 
  \end{equation}
  By the same arguments as in \eqref{2.3}--\eqref{2.5}, we can express \eqref{3.1} equivalently as $\pi(t,i) = \Gamma_{\lambda}(V_{\lambda}^{\pi})(t,i)$ for all $(t,i)\in \N_0\times\Sc$, which amounts to $\pi = \Phi_{\lambda}(\pi)$.

  
  If $v \in \Xc$ satisfies that $v = V_{\lambda}^{\pi}$ for a regularized equilibrium $\pi \in \Pi_r$, then by $\pi = \Phi_{\lambda}(\pi) = \Gamma_{\lambda}(V_{\lambda}^{\pi})$, $v = V_{\lambda}^{\pi} = V_{\lambda}^{\Gamma_{\lambda}(v)} = \Psi_{\lambda}(v)$. Conversely, if $v\in \Xc$ satisfies that $v = \Psi_{\lambda}(v)\in\Xc$, set $\pi := \Gamma_{\lambda}(v) \in \Pi_r$. Then, $v = V_{\lambda}^{\pi}$ and thus $\pi = \Gamma_{\lambda}(v) = \Gamma_{\lambda}(V_{\lambda}^{\pi}) = \Phi_{\lambda}(\pi)$. This implies that $\pi$ is a regularized equilibrium. 
\end{proof}

Next we show the existence of fixed points for $\Psi_\lambda$, which tells the existence of regularized equilibria,

\begin{Theorem}\label{thm:existence.regular}
Assume that Assumptions~\ref{assum:reward-bound}--\ref{assum:cone} hold. 
\bi  
\item[(i)] For every $\lambda>0$, there exists a fixed point of $\Psi_\lambda$, and hence a regularized equilibrium $\pi^*\in\Pi_r$ for the entropy-regularized problem \eqref{1.5} exists.
 
\item[(ii)] Moreover, there exists a finite constant $M^*$ independent of $\lambda\in(0,1]$ such that, for each $\lambda\in(0,1]$, a regularized equilibrium $\pi$ exists with its auxiliary function $V^\pi_\lambda$ belonging to $\Kc(M^*)$. 
\ei 
\end{Theorem}


\subsection{Proof of Theorem \ref{thm:existence.regular}}

To begin with, we first improve the entropy estimate given in \cite[Lemma 1]{bayraktar2025relaxed} for finite-state Markov chain in bounded case, as a refined version, to the current unbounded case. 

\begin{Lemma}\label{lm:entropy.est}
Given an arbitrary $M\in(0,\infty)$ and consider the subset $\Kc(M)$. Then for every $\lambda>0$, there exist finite constants $C_0, C_1,C_2, C_3>0$, independent of $\lambda$, $M$ and $(t,i)\in \N_0\times \Sc$, such that
$$
|\mathcal H(\Gamma_\lambda(t, v(\cdot; t,i),i))| \leq  C_0 +C_1|\ln(\lambda)| +C_2\ln(1+M) +C_3 w(i),\quad \forall v\in \Kc(M).
$$
\end{Lemma}
\begin{proof} 
Take an arbitrary $M\in(0,\infty)$ and consider $v\in \Kc(M)$. By \eqref{eq: assum.wW0} and the definition of $\Kc(M)$ in \eqref{eq:def.KM},
\be\label{eq:entropyest.0}  
\|v(\cdot; t, i)\|_{\wt}= \sup_{j\in \Sc} \frac{|v(j;t,i)|}{W(j)}\leq \sup_{j\in \Sc}\frac{M(w(i)+w(j))}{W(j)}\leq M(w(i)+C_w).
\ee 

Fix a state $(t, i)\in \N_0\times\Sc$ and a reference point $\bar{a} \in A$. Recall \eqref{eq:def.H}.
For any $a\in A$, by \eqref{eq:assume.lipaf}, \eqref{eq:assume.lipap}, and then \eqref{eq:entropyest.0}, we have
\be\label{eq:entropyest.1} 
\begin{aligned}
&|H(t,i, v(\cdot ;t,i), a) - H(t,i, v(\cdot ;t,i), \bar a)| \\
& \leq |f(t,i,a; t,i) - f(t,i,\bar{a}; t, i)| + \sum_{j \in S} |p(t,i,j,a)-p(t,i,j,\bar a)| \, |v(j;t, i)| \\
& \leq L_f w(i) |a - \bar{a}| +  \sum_{j \in S} |p(t,i,j,a)-p(t,i,j,\bar a)| \, W(j) \|v(\cdot; t, i)\|_{\wt} \\
& \leq L_f w(i) |a - \bar{a}| + L_p W(i) |a - \bar{a}|  \|v(\cdot; t, i)\|_{\wt} \\
& \leq \left[  L_f w(i)  +  L_p M(w(i)+C_w) W(i)  \right] |a - \bar{a}|.
\end{aligned}
\ee
The Gibbs distribution is given by
\bee
\Gamma_{\lambda}(t, v(\cdot; t, i),i)(a) = \frac{\exp\left(\frac{1}{\lambda} H(t, i, v(\cdot; t, i), a)\right)}{Z},
\quad \text{where } Z := \int_{A} \exp\left(\frac{1}{\lambda} H( t, i, v(\cdot; t, i),a')\right) da'. 
\eee
Set $L_{M,i}:= L_f w(i)  +  L_p M(w(i)+C_w) W(i)$. From \eqref{eq:entropyest.0}, for any $a \in A$,  
$$H(t, i, v(\cdot; t, i), a) \geq H(t, i, v(\cdot; t, i), \bar{a}) - L_{M,i} |a - \bar{a}|,$$
which gives
\begin{equation}\label{3.3}
Z \geq \exp\left(\frac{1}{\lambda} H(t, i, v(\cdot; t, i), \bar{a}) \right) \int_{A} \exp\left(-\frac{1}{\lambda} L_{M,i} |a' - \bar{a}|\right) da'.   
\end{equation}
Thus, the Gibbs density at $\bar a$ satisfies:
$\Gamma_{\lambda}(y,i)(\bar{a}) \leq \frac{1}{\int_{A} \exp\left(-\frac{1}{\lambda} L_{M, i} |a - \bar{a}|\right) da}$.

Let us first derive an estimate for dimension $\ell>1$. By Assumption~\ref{assum:cone}, there exists a cone $\textsf{cone}_\iota(\bar{a})$ such that  
\begin{equation}\label{3.4}
\int_{A} e^{-\frac{1}{\lambda} L_{M, i} |a - \bar{a}|} da \geq \int_{\textsf{cone}_\iota(\bar{a})\cap B_{\theta}(\bar{a})} e^{-\frac{1}{\lambda} L_{M, i} |a - \bar{a}|} da. 
\end{equation}
By applying a translation and rotation and letting $u = a - \bar{a}$, we have
\begin{equation}\label{3.5}
\int_{\textsf{cone}_\iota(\bar{a})\cap B_{\theta}(\bar{a})} e^{-\frac{1}{\lambda} L_{M, i} |a - \bar{a}|} da = \int_{\textsf{cone}_\iota \cap B_{\theta}(0)} e^{-\frac{1}{\lambda} L_{M, i} |u|} du
\end{equation}
where $\textsf{cone}_\iota$ is a cone with vertex at the origin and angle $\iota$. We now estimate the integral on the right-hand side of \eqref{3.5} in two cases:\\
\textbf{Case 1:} $\frac{1}{\lambda} L_{M, i} \theta \leq 1$. For $u \in \textsf{cone}_\iota \cap B_{\theta}(0)$, we have $|u| \leq \theta$, so $e^{-\frac{1}{\lambda} L_{M, i}|u|} \geq e^{-1}$. Therefore,  
\begin{equation}\label{3.6}
\int_{\textsf{cone}_\iota \cap B_{\theta}(0)} e^{-\frac{1}{\lambda} L_{M, i} |u|} du \geq e^{-1} \mathrm{Leb}(\textsf{cone}_\iota \cap B_{\theta}(0)) =: K_{0} 
\end{equation}
\textbf{Case 2:} $\frac{1}{\lambda} L_{M, i} \theta > 1$. Using $\ell$-dimensional spherical coordinates, denote  $K_{1} := \int_{\textsf{cone}_\iota \cap S^{\ell - 1}} d\omega >0 \quad (\text{solid angle constant}), \quad K_{2} := \int_{0}^{1} z^{\ell - 1} e^{-z} dz > 0.$ Then  
\be\label{3.7}
\begin{aligned}
\int_{\textsf{cone}_\iota \cap B_{\theta}(0)} e^{-\frac{1}{\lambda} L_{M, i} |u|} du &= K_{1} \int_{0}^{\theta} r^{\ell - 1} e^{-\frac{1}{\lambda} L_{M, i} r} dr \\
&= K_{1} \left(\frac{1}{\lambda} L_{M, i}\right)^{-\ell} \int_{0}^{\frac{1}{\lambda} L_{M, i} \theta} z^{\ell - 1} e^{-z} dz \quad (z = \frac{1}{\lambda} L_{M, i} r) \\
&\geq K_{1} \left(\frac{1}{\lambda} L_{M, i}\right)^{-\ell} \int_{0}^{1} z^{\ell - 1} e^{-z} dz = K_{1} K_{2} \left(\frac{1}{\lambda} L_{M, i}\right)^{-\ell}.
\end{aligned}
\ee
Combining both cases, we obtain  
$$\int_{A} e^{-\frac{1}{\lambda} L_{M, i}  |a - \bar{a}|} da \geq \min \left\{ K_{0}, \, K_{1} K_{2}  \left(\frac{1}{\lambda} L_{M, i}\right)^{-\ell} \right\}.$$
We then conclude from \eqref{3.6} and \eqref{3.7} that, for any $a \in A$,  
\begin{equation*}
\begin{aligned}
\Gamma_{\lambda}(t, v(\cdot; t, i),i)(a)&  \leq \max \left\{ \frac{1}{K_{0}}, \, \frac{1}{K_{1} K_{2}} \left(\frac{1}{\lambda} L_{M, i}\right)^{\ell}  \right\} \\
&=  \max \left\{ \frac{1}{K_{0}}, \, \frac{1}{K_{1} K_{2}}\lambda^{-\ell}\left[ L_f w(i)  +  L_p M(w(i)+C_w) W(i) \right]^{\ell}  \right\}\\
&\leq  \max \left\{ \frac{1}{K_{0}}, \, \frac{1}{K_{1} K_{2}}\lambda^{-\ell}\left[ \left(L_f+L_p\right)\left(1+C_w\right)(1+M)w(i) W(i) \right]^{\ell}  \right\}.
\end{aligned}
\end{equation*}
Notice that $w(i), W(i) \geq 1$. Taking logarithms of above inequality gives
\be\label{eq:lm.Hest} 
\begin{aligned}
\ln(\Gamma_{\lambda}(t, v(\cdot; t, i),i)(a))\leq &\max\left\{|-\ln(K_0)|, \ell\left| \ln\left(\frac{(L_f +L_p)(1+C_w)}{K_1K_2} \right) \right| \right\}\\
&+\ell |\ln\lambda|+\ell \ln(1+M)+\ell\ln(W(i)) + \ell \ln(w(i))\\
\leq &\max\left\{|-\ln(K_0)|, \ell\left| \ln\left(\frac{(L_f +L_p)(1+C_w)}{K_1K_2} \right) \right| \right\}\\
&+\ell |\ln\lambda|+\ell \ln(1+M)+\ell (C_{W}+1)w(i),
\end{aligned}
\ee 
where the last inequality follows from the second line in \eqref{eq: assum.wW}.
Hence, $\Hc(\Gamma_{\lambda}(t, v(\cdot; t, i),i))$ is bounded by the right-hand side of the above inequality.

For $\ell=1$, by following arguments similar to the above, with $\textsf{cone}_\iota \cap B_{\theta}(0)$, $K_0$, $K_1$, and $K_2$ replaced by $[0,\zeta]$, $e^{-1}\zeta$, $1$, and $\int_0^1 e^{-z}dz=1-e^{-1}$, respectively, we also obtain the upper bound in \eqref{eq:lm.Hest} .

By the principle of maximum entropy, the uniform distribution maximizes entropy:  
\begin{equation}\label{3.10}
\int_{A} \Gamma_{\lambda}(v)(t, i, a) \ln \Gamma_{\lambda}(v)(t,i,a) da \geq -\ln (\Leb(A)).
\end{equation}
Thus, the desired result follows.

\end{proof}

Now we provide certain estimates which shall be used to show the self mapping of $\Psi$.
\begin{Lemma}\label{lm:est.0}
For every $\pi\in\Pi$ and $(t,i)\in\mathbb N_0\times \Sc$, the following moment estimate holds:
\be\label{eq:lm0.estWX} 
\left| \E[w(X^\pi_s)\mid X^\pi_t=i] \right|
\leq \gamma^{s-t}w(i)+\frac{b_w}{1-\gamma_w},\quad \forall s\geq t.
\ee 
Consequently,
the reward term satisfies, for any $0\leq t< s$ and any $j\in\Sc$ that
\be\label{eq:lm0.estfX} 
\left| \E\left[f^{\pi(s,X^\pi_{s})}(s,
X_{s};t,i)\mid X^\pi_{t+1}=j\right]\right|
\leq C_f\delta(s-t)\left( w(i) + \gamma_w^{s-t-1}w(j)+\frac{b_w}{1-\gamma_w} \right).
\ee
\end{Lemma}
\begin{proof}
Note that $\gamma_w \in(0,1)$ and $w(i)\geq 1$, by applying \eqref{eq:assume.pw} in Assumption \ref{assum:lyapunov.p} $s-t$ times we have
\begin{align*}
\E [w(X^\pi_{s}) \mid  X^\pi_t =i]  
 &\leq \gamma^{s-t} w(i) + b_w\sum_{k=0}^{s-t-1} \gamma_w^k\leq  \gamma_w^{s-t} w(i) + \frac{b_w}{1-\gamma_w},
\end{align*}
Consequently,
$\frac{\mathbb{E}[w(X^\pi_{s}) | X^\pi_t =i]}{w(i)} \leq \gamma_w^{s-t} + \frac{b_w}{(1-\gamma_w)w(i)} \leq \gamma_w^{s-t} +  \frac{b_w}{1-\gamma_w}.$
By \eqref{eq:assume.fgrow}, for any $0\leq t<s$ and $(i, j) \in \Sc^2$, 
$$|f^{\pi(s,j)}(s,j;t,i)| = \left| \int_{A} f(s,j,a;t,i) \pi(s,j)(da) \right| \leq C_{f} \delta(s-t) [w(i)+w(j)].$$
Then the expectation of the reward part is estimated as
\begin{equation}\label{3.13}
  \begin{aligned}
\mathbb{E}\left[ f^{\pi(s, X^\pi_{s})}(s,X^\pi_{s}; t, i) | X^\pi_{t+1}=j\right] \leq &C_{f} \delta(s-t) \big[w(i)+\mathbb{E}[w(X^\pi_{s}) |X^\pi_{t+1} =j] \big] \\
\leq & C_{f} \delta(s-t) \left[  w(i)+\gamma_w^{s-t-1}w(j)+\frac{b_w}{1-\gamma_w} \right]. 
  \end{aligned}
\end{equation}
\end{proof}

Based on the above preparation, we are now ready to show $\Psi$ maps from $\Kc(M)$ to $\Kc(M)$ for $M$ sufficiently large.

\begin{Lemma}[Self-mapping of $\Psi_\lambda$]\label{lm:self-mapping}
Fix an arbitrary $\lambda\in(0,\infty)$, there exists a finite constant $M^*_\lambda>0$ dependent on $\lambda$ such that for any $M\in(M^*_\lambda,\infty)$,
$$
\Psi_\lambda(\Kc(M))\subseteq \Kc(M).
$$
And, if $v\in \Kc(M)$ for some $M\geq M^*_\lambda$ is a fixed point of $\Psi_\lambda$, then $v$ belongs to $\Kc(M^*_\lambda)$.

Moreover, there exists a finite constant $M^*>0$ such that the above statement holds uniformly over all $\lambda\in (0,1]$.
\end{Lemma}

\begin{proof}
Let $M\in(0,\infty)$ and take $v(j;t,i) \in \Kc(M)$ and set $\pi = \Gamma_{\lambda}(v)$. From the expression of the value function we have
\begin{equation}\label{3.14}
|\Psi_{\lambda}(v)(j; t, i)| \leq \mathbb{E}\left[ \sum_{k=t+1}^{\infty} \left( |f^{\pi(k, X^\pi_{k})}(k, X_{k}^{\pi}; t, i)| + |\lambda \delta(k-t) \mathcal{H}(\pi(k, X^\pi_{k}))| \right) \bigg| X^\pi_{t+1}= j\right] 
\end{equation}
By Lemma~\ref{lm:entropy.est}, \eqref{eq:lm0.estWX} in Lemma \ref{lm:est.0} and that $\ln(x)\leq x$ on $[1,\infty)$,
\begin{equation}\label{3.15}
\begin{aligned}
\E[|\mathcal{H}(\pi(k,X^\pi_k))| \mid X^\pi_{t+1}=j]\leq &C_0 +C_1|\ln(\lambda)| +C_2\ln(1+M) +C_3 \ln \E[w(X^\pi_k) \mid X^\pi_{t+1}=j]\\
\leq & C_0 +C_1|\ln(\lambda)| +C_2\ln(1+M) +C_3 \left( \gamma_w^{k-t-1}w(j)+\frac{b_w}{1-\gamma_w}\right)
\end{aligned}
\end{equation}
By \eqref{eq:lm0.estfX}, we have for each $k=t+1, t+2,\cdots$ that
\begin{equation}\label{3.16}
\E\left[ f^{\pi(k,X_{k})}(k, X_{k}; t, i) | X^\pi_{t+1}=j \right] \leq C_{f} \delta(k-t) \left( w(i)+ \gamma_w^{k-t-1}w(j) +\frac{b_w}{1-\gamma_w} \right).
\end{equation}
By plugging  \eqref{3.15} and \eqref{3.16} into \eqref{3.14} and noticing $w\geq 1$, we have
\begin{align*}\label{3.17}
&|\Psi_{\lambda}(v)(j;t, i)| \\
&\leq \sum_{k=t+1}^{\infty} C_{f} \delta(k-t) \left( w(i)+ \gamma_w^{k-t}w(j) +\frac{b_w}{1-\gamma_w} \right)\\
&\quad+\sum_{k=t+1}^\infty \lambda \widetilde\delta(k-t) \left[ C_0 +C_1|\ln(\lambda)| +C_2\ln(1+M) +C_3 \left( \gamma_w^{k-t-1}w(j)+\frac{b_w}{1-\gamma_w}\right)\right]\\
&\leq \sum_{k=1}^{\infty} (\delta(k)+ \widetilde\delta(k)) \bigg[(C_{f}+\lambda C_3)\frac{b_w}{1-\gamma_w} +\lambda C_0+C_1\lambda|\ln(\lambda)| \bigg]+\lambda C_2 \sum_{k=1}^{\infty} \widetilde\delta(k) \ln(1+M)\\
&\quad + C_f \sum_{k=1}^\infty\delta(k) w(i) +( C_f +\lambda C_3)\left[\sum_{k=0}^\infty[\delta(k)+\widetilde\delta(k)]\gamma_w^k\right] w(j)\\
&\leq \widetilde C\big(1+ \lambda+\lambda |\ln\lambda| \big)\ln(2+M)\big[w(i)+w(j) \big],
\end{align*}
where the last inequality follows from the assumptions that $\sum_{k=0}^\infty ((\delta(k)+\tilde \delta(k))<\infty$ and $\gamma_w<1$, and $\widetilde C$ is a finite constant independence of $\lambda \in (0,1)$, $i, j,$ and $M$. The function $\theta(M): (0,\infty) \to \R$ defined by
\begin{equation}\label{3.18}
\theta(M):=\widetilde C\big(1+ \lambda+\lambda |\ln\lambda| \big)\ln(2+M) 
\end{equation}
is of sublinear growth. So there exist  a finite positive constant $M^*_\lambda$ such that 
\be\label{eq:theta.M} 
\theta(M^*_\lambda) = M^*_\lambda, \quad\text{and }\theta(M)< M \; \; \forall M  \in(M^*_\lambda,\infty).
\ee 
That is, for any $M\geq M^*_\lambda$,
$$
|\Psi_{\lambda}(v)(j;t, i)|\leq M (w(i)+w(j))\quad \forall (j,t,i)\in \Sc\times\N_0\times \Sc, \text{ provided that } v\in \Kc(M).
$$ 
Also, it is clear from the above inequality that $\|\Psi_{\lambda}(v)(\cdot ;t, i)\|_{\wt}<\infty$ for each $(t,i)\in \Sc\times \N_0$. Thus, $\Psi_{\lambda}(v)\in \Xc$. 
Therefore, $\Psi(\Kc(M))\subset \Kc(M)$, for any $M\in[M^*_\lambda,\infty)$.  Suppose $v\in \Kc(M)$ for some $M\geq M^*_\lambda$ is a fixed point of $\Psi_\lambda$. Then $\theta(M)<M$ in \eqref{eq:theta.M} also implies that $\Psi_\lambda(v)\in \Kc(M^*_\lambda)$.

Moreover, by $\sup_{\lambda\in(0,1]}\lambda|\ln\lambda| =1/e$, the fixed point of $\theta$ in \eqref{eq:theta.M} can be chosen as a finite constant $M^*$ uniformly over all $\lambda\in(0,1]$.
\end{proof}

\begin{Lemma}[Continuity of $\Psi_\lambda$]\label{lm:Psi.continuous}
Fix $\lambda>0$ and let $M\geq M^*_\lambda$, where $M^*_\lambda$ is given by Lemma~\ref{lm:self-mapping}. Then the mapping
$
\Psi_\lambda:\Kc(M)\to  \Kc(M)
$
is continuous under the topology inherited from $\Xc$. More precisely, if $v^n,v\in \Kc(M)$ and $v^n\to v$ in $\Xc$, then, for every $(t,i)\in\N_0\times\Sc$,
$$
\big\|
\Psi_\lambda(v^n)(\cdot;t,i)-\Psi_\lambda(v)(\cdot;t,i)\big\|_{\wt}\to 0.
$$
\end{Lemma}

\begin{proof}
Let $(v^n)_{n\geq1}\subset \Kc(M)$ and $v\in \Kc(M)$ satisfy that $v^n\to v$ in the product topology of $\Xc$. That is,
\be\label{eq:product.conv.coordinate}
\|v^n(\cdot;t,i)-v(\cdot; t, i)\|_{\wt}\to 0,\quad \forall (t, i)\in\N_0\times\Sc.
\ee
Define $\pi^n:=\Gamma_\lambda(v^n)$ and $\pi:=\Gamma_\lambda(v)$. We divide the proof into five steps.

\textbf{Step 1.} We first prove the convergence of $\pi^n$ to $\pi$. Recall \eqref{eq:def.H}. Fix an arbitrary $(t,i)\in\N_0\times\Sc$, 
By \eqref{eq:assume.pW} in Assumption \ref{assum:lyapunov.p},
$$
\begin{aligned}
&\sup_{a\in A}|H(t, i, v^n(\cdot; t, i), a)-H(t, i, v(\cdot; t, i), a)|\\
&\leq \|v^n(\cdot;t,i)-v(\cdot;t,i)\|_{\wt} \sup_{a\in A}\sum_{j\in\Sc}p(t,i,j,a)W(j)
\leq \big(\gamma_{W}W(i)+b_{W}\big)\|v^n(\cdot;t,i)-v(\cdot;t,i)\|_{\wt}.
\end{aligned}
$$
It then follows from \eqref{eq:product.conv.coordinate} that
\be\label{eq:F.uniform.convergence}
\|H(t, i, v^n(\cdot; t, i), a)-H(t, i, v(\cdot; t, i), a)\|_{L^\infty(A)}\to 0.
\ee
The above uniform convergence on the compact set $A$ also implies the uniform boundedness of $(H(t, i, v^n(\cdot; t, i), a))_{n\in \N}$ and  $H(t, i, v(\cdot; t, i), a)$ on $A$, which is independent of $t, i$. Note that 
$$
\pi^n(t,i,a)=\frac{\exp(H(t, i, v^n(\cdot; t, i), a)/\lambda)}{\int_A\exp(H(t, i, v^n(\cdot; t, i), a')/\lambda)\,da'},\quad
\pi(t,i,a)=\frac{\exp(H(t, i, v(\cdot; t, i), a)/\lambda)}{\int_A\exp(H(t, i, v(\cdot; t, i), a')/\lambda)\,da'},
$$
Thus, there exist constants
$0<c_{t,i}\leq C_{t,i}<\infty$, independent of $n$, such that
$$
c_{t,i} \leq \pi^n(t,i,a),\,\pi(t,i,a)\leq C_{t,i},\quad\forall a\in A \text{ and } n\geq1.
$$
And the uniform convergence in \eqref{eq:F.uniform.convergence} further implies
\be\label{eq:gibbs.uniform.convergence}
\sup_{a\in A}|\pi^n(t,i,a)-\pi(t,i,a)|\to 0, \text{ and hence, }
\|\pi^n(t,i)-\pi(t,i)\|_{\mathrm{TV}}\to 0.
\ee 
The function $x\mapsto x\ln x$ is uniformly continuous on
$[c_{s,i},C_{s,i}]$. Hence \eqref{eq:gibbs.uniform.convergence} also yields
\begin{equation}\label{eq:entropy.node.convergence}
|\Hc(\pi^n(t,i))-\Hc(\pi(t,i))|\to 0.
\end{equation}
Moreover, by Assumption \ref{assum:lipschitz}, \eqref{eq:gibbs.uniform.convergence} also gives, for every $(t,s)\in \Delta$ and every $i,j\in \Sc$ that
\begin{align}
p^{\pi^n}(t,i,j)&=\int_A p(t,i,j,a)\pi^n(t,i,a)\,da \to p^\pi(t,i,j),\label{eq:transition.node.convergence}\\  f^{\pi^n(s,i)}(s,j;t,i)&=  \int_A f(s,j,a;t,i)\pi^n(s,j,a)\,da \to f^{\pi(s,j)}(s,j;t,i). \label{eq:reward.node.convergence}
\end{align}

\medskip
\noindent
\textbf{Step 2.} We next show the convergence of multi-step transition probabilities. Fix an arbitrary $j\in \Sc$. For $m\geq1$ and $r\in\Sc$, set
$$
q^n_{m}(r):=\P\big(X_{t+m}^{\pi^n}=r\,\big|\,X_{t+1}^{\pi^n}=j\big),
\quad
q_m(r):=\P\big(X_{t+m}^{\pi}=r\,\big|\,X_{t+1}^{\pi}=j\big).
$$
We first verify by induction that
\be\label{eq:lmcontinue.step2.1}
q^n_{m}(r)\to  q_m(r),\quad
\forall r\in\Sc,\; \forall m\geq1.
\ee
\eqref{eq:lmcontinue.step2.1} holds immediately for $m=1$. Suppose that the assertion holds for
some $m\geq1$. By the Markov property,
$$
q^n_{m+1}(r)=\sum_{l\in\Sc}q^n_{m}(l)p^{\pi^n}(t+m,l,r),\quad 
q_{m+1}(r)=\sum_{l\in\Sc}q_m(l)p^\pi(t+m,l,r).
$$
Fix $R\in\N$. Splitting the state sum gives that
$$
\begin{aligned}
|q^n_{m+1}(r)-q_{m+1}(r)|&\leq\sum_{l\leq R}  \left|  q^n_{m}(l)  p^{\pi^n}(t+m,l,r)  -  q_m(l)  p^\pi(t+m,l,r)  \right|\\
&\quad+  \sum_{l>R}q^n_{m}(l)  p^{\pi^n}(t+m,l,r)+  \sum_{l>R}q_m(l)p^\pi(t+m,l,r).
\end{aligned}
$$
For each fixed $R$, the first term on the right-hand side converges to zero by the induction
hypothesis and \eqref{eq:transition.node.convergence}. Since
$p^{\pi^n}(t+m,l,r)\leq1$ and
$p^\pi(t+m,l,r)\leq1$, the two tail terms containing summations over $l>R$ are bounded by
$
\sum_{l>R}q^n_{m}(l)$ \text{and} $\sum_{l>R}q_m(l),
$
respectively. By iterating \eqref{eq:assume.pW},
\be\label{eq:lmcontinue.step2.3}
\sup_{n\geq1}\E\left[W(X_{t+m}^{\pi^n})\,|\,X_{t+1}^{\pi^n}=j\right]\leq\gamma_{W}^{m-1}W(j)
+\frac{b_{W}}{1-\gamma_{W}}.
\ee
The same estimate holds with $\pi^n$ replaced by $\pi$. Since
$W(l)\to\infty$ as $l\to\infty$,
$$
\begin{aligned}
\sup_{n\geq1}\sum_{l>R}q^n_{m}(l)&\leq\frac{1}{\inf_{l>R}W(l)}\sup_{n\geq1}\sum_{l>R}q^n_{m}(l)W(l)
&\leq\frac{  \gamma_{W}^{m-1}W(j)+  \frac{b_{W}}{1-\gamma_{W}}  }{  \inf_{l>R}W(l)  }  \to 0, \text{ as $R\to\infty$. }
\end{aligned}
$$
The same conclusion holds for
$\sum_{l>R}q_m(l)$. Hence, letting first $n\to\infty$ and then $R\to\infty$ proves the induction step, and \eqref{eq:lmcontinue.step2.1} is verified.

\medskip
\noindent
\textbf{Step 3.}
We next establish the uniform integrability needed to pass to the
limit in the reward and entropy expectations. From
\eqref{eq:lmcontinue.step2.3} and the first line of
\eqref{eq: assum.wW},
\begin{align}
&\sup_{n\geq1}\E\left[w(X_{t+m}^{\pi^n})\mathbf 1_{\{X_{t+m}^{\pi^n}>R\}}  \,\middle|\,  X_{t+1}^{\pi^n}=j  \right]  
\leq   \sup_{r>R}\frac{w(r)}{W(r)}\sup_{n\geq1}  \E\left[  W(X_{t+m}^{\pi^n})  \mathbf 1_{\{X_{t+m}^{\pi^n}>R\}}  \,\middle|\,  X_{t+1}^{\pi^n}=j  \right]  \nonumber\\
& \leq  \sup_{r>R}\frac{w(r)}{W(r)}  \left(  \gamma_{W}^{m-1}W(j)  +  \frac{b_{W}}{1-\gamma_{W}}  \right)  \to 0, \; \text{ as }R\to\infty.  \label{eq:w.uniform.integrability}
\end{align}
A similar conclusion also holds with $\pi^n$ replaced by $\pi$.

Since $v^n,v\in\Kc(M)$, Lemma~\ref{lm:entropy.est} yields a
constant $C_{\lambda,M}>0$ exists 
such that
\begin{equation}\label{eq:entropy.growth.uniform}
 |\Hc(\pi^n(s,r))| + |\Hc(\pi(s,r))|  \leq  C_{\lambda,M}w(r),\quad \forall n\in \N_0, s\in\N_0, r\in \Sc.
\end{equation}
Similarly, Assumption~\ref{assum:reward-bound} gives
\be\label{eq:reward.growth.uniform}
  \left|
  f^{\pi^n(s,r)}(s,r;t,i)
  \right|
  +
  \left|
  f^{\pi(s,r)}(s,r;t,i)
  \right|
  \leq
  2C_f\delta(s-t)\big(w(i)+w(r)\big).
\ee
Define
$$
g^n_{m}(r):=f^{\pi^n(t+m,r)}(t+m,r;t,i),\qquad
g_m(r):=f^{\pi(t+m,r)}(t+m,r;t,i).
$$
By \eqref{eq:reward.node.convergence},
$
g^n_{m}(r)\to  g_m(r)
$, for all $r\in \Sc$.
Moreover,
$$
\begin{aligned}
&\E\left[  f^{\pi^n(t+m,X_{t+m}^{\pi^n})}  (t+m,X_{t+m}^{\pi^n};t,i)  \,\middle|\,  X_{t+1}^{\pi^n}=j  \right] =  \sum_{r\in\Sc}q^n_{m}(r)g^n_{m}(r),
\end{aligned}
$$
and the analogous identity holds for $\pi$. For $R\in\N$, we have
\be\label{eq:reward.step3} 
\begin{aligned}
&\left|  \sum_{r\in\Sc}q^n_{m}(r)g^n_{m}(r)  -  \sum_{r\in\Sc}q_m(r)g_m(r)  \right|  \\
&\leq  \sum_{r\leq R}  \left|  q^n_{m}(r)g^n_{m}(r)  -  q_m(r)g_m(r)  \right|
+  \sum_{r>R}q^n_{m}(r)|g^n_{m}(r)|  +  \sum_{r>R}q_m(r)|g_m(r)|.
\end{aligned}
\ee
For fixed $R$, \eqref{eq:lmcontinue.step2.1} and
\eqref{eq:reward.node.convergence} together tells the first term on the right-hand side of \eqref{eq:reward.step3} converges to zero as $n\to\infty$.
 By \eqref{eq:reward.growth.uniform} and $w(r)\geq1$,
$$
\begin{aligned}
\sum_{r>R}q^n_{m}(r)|g^n_{m}(r)|&\leq  C_f\delta(m)  \sum_{r>R}q^n_{m}(r)\big(w(i)+w(r)\big)  \\
  &\leq C_f\delta(m)(1+w(i))  \sum_{r>R}q^n_{m}(r)w(r) \to 0, \; \text{uniformly in $n$,  as }R\to \infty,
\end{aligned}
$$
where the last line above follows from \eqref{eq:w.uniform.integrability}.  The third term on the right-hand side of \eqref{eq:reward.step3} is
treated in exactly the same way. We therefore obtain from \eqref{eq:reward.step3} that
\be\label{eq:fconverge.step3} 
\begin{aligned}
&\E\left[  f^{\pi^n(t+m,X_{t+m}^{\pi^n})}  (t+m,X_{t+m}^{\pi^n};t,i)  \,\middle|\,  X_{t+1}^{\pi^n}=j  \right]  \\
&\to   \E\left[  f^{\pi(t+m,X_{t+m}^{\pi})}  (t+m,X_{t+m}^{\pi};t,i)  \,\middle|\,  X_{t+1}^{\pi}=j  \right],\quad \text{as $n\to\infty$}.
\end{aligned}
\ee 

We proceed similarly for the entropy term. Set
$$
h^n_{m}(r):=\Hc(\pi^n(t+m,r)),\quad 
h_m(r):=\Hc(\pi(t+m,r)).
$$
By \eqref{eq:entropy.node.convergence},
$
h^n_{m}(r)\to  h_m(r),
$
for all $r\in \Sc$. By applying \eqref{eq:entropy.growth.uniform}, we have that
$$
\begin{aligned}
&  \left|  \sum_{r\in\Sc}q^n_{m}(r)h^n_{m}(r)  -  \sum_{r\in\Sc}q_m(r)h_m(r)  \right|  \\
&\leq  \sum_{r\leq R}  \left|  q^n_{m}(r)h^n_{m}(r)  -  q_m(r)h_m(r)  \right|  +  C_{\lambda,M}  \sum_{r>R}q^n_{m}(r)w(r)
+  C_{\lambda,M}  \sum_{r>R}q_m(r)w(r).
\end{aligned}
$$
For fixed $R$, the summation in the second line above converges to zero as $n\to\infty$. Hence, \eqref{eq:w.uniform.integrability} tells that the first tail term in the third line above converges to zero as $R\to\infty$, so does the second tail term. In sum,
\be\label{eq:Hconverge.step3}  
\begin{aligned}
&\E\left[  \Hc(\pi^n(t+m,X_{t+m}^{\pi^n}))  \middle|  X_{t+1}^{\pi^n}=j  \right]\to   \E\left[  \Hc(\pi(t+m,X_{t+m}^{\pi}))  \middle|  X_{t+1}^{\pi}=j  \right], \; \text{as $n\to\infty$}.
\end{aligned}
\ee 

\medskip
\noindent
\textbf{Step 4.}
We now prove the pointwise convergence of $\Psi_\lambda(v^n)(j;t,i)$ to $\Psi_\lambda(v)(j;t,i)$. Fix $(t,i,j)\in\N_0\times\Sc^2$. For $N\in\N$, define
$$
\begin{aligned}
V_{\lambda,N}^{\pi^n}(j;t,i)  :=\E\Bigg[  \sum_{k=t+1}^{t+N}  \Big(  &f^{\pi^n(k,X_k^{\pi^n})}  (k,X_k^{\pi^n};t,i)+  \lambda\widetilde\delta(k-t)  \Hc(\pi^n(k,X_k^{\pi^n}))  \Big)  \,\Bigm|\,  X_{t+1}^{\pi^n}=j  \Bigg],
\end{aligned}
$$
and define $V_{\lambda,N}^{\pi}(j;t,i)$ in a similar way. Then for every fixed $N\in\N$, summing the convergences in \eqref{eq:fconverge.step3} and $\lambda\delta(m)$ times the entropy convergence in \eqref{eq:Hconverge.step3}  over
$m=1,\ldots,N$ gives
\be\label{eq:lmcontinue.step2.0}
\lim_{n\to\infty}V_{\lambda,N}^{\pi^n}(j;t,i) = V_{\lambda,N}^{\pi}(j;t,i).
\ee

We next pass from the truncated values to the infinite-horizon values. By Lemmas \ref{lm:entropy.est} and \ref{lm:est.0}, there exists a constant $\widetilde C_{\lambda,M}>0$, independent of $n,t,i,j,N$, such that
\begin{align}
&\left|  V_\lambda^{\pi^n}(j;t,i)  -  V_{\lambda,N}^{\pi^n}(j;t,i)  \right| \notag\\
&\leq  \widetilde C_{\lambda,M}  \sum_{m=N+1}^{\infty} ( \delta(m)+\widetilde\delta(m))  \left(  1+w(i)+\gamma_w^{m-1}w(j)  +\frac{b_w}{1-\gamma_w}  \right).  \label{eq:value.temporal.tail}
\end{align}
The same estimate holds again with $\pi^n$ replaced by $\pi$. Since
$
\sum_{m=0}^{\infty}[\delta(m)+\widetilde\delta(m)]<\infty
$
and $\gamma_w\in(0,1)$, the right-hand side of \eqref{eq:value.temporal.tail} converges to zero as $N\to\infty$, for every fixed $(t,i,j)$, uniformly in $n$. Combining this with \eqref{eq:lmcontinue.step2.0} tells that
\be\label{eq:value.pointwise.convergence}
V_\lambda^{\pi^n}(j;t,i)  \to   V_\lambda^\pi(j;t,i),  \quad  \forall(t,i,j)\in\N_0\times\Sc^2.
\ee

\medskip
\noindent
\textbf{Step 5.} Fix an arbitrary $(t,i)\in\N_0\times\Sc$. Now we wrap up the proof by showing 
\be\label{eq:step3}
\|\Psi_\lambda(v^n)(\cdot; t, i)- \Psi_\lambda(v)(\cdot; t, i)\|_{\wt}\to 0.
\ee
By Lemma \ref{lm:self-mapping},
$
\Psi_\lambda(v^n),\Psi_\lambda(v)\in \Kc(M).
$
Therefore,
\begin{align}
\frac{|\Psi_\lambda(v^n)(j;t,i)-\Psi_\lambda(v)(j;t,i)|}{W(j)}\leq 2M\left(\frac{w(i)}{W(j)} +\frac{w(j)}{W(j)}  \right).  \label{eq:uniform.spatial.tail}
\end{align}
It then follows from \eqref{eq: assum.wW} that the right-hand side of
\eqref{eq:uniform.spatial.tail} converges to zero as $j\to\infty$, uniformly in $n$. Hence, for every $\eps>0$, there exists a finite set $F_{\eps, t, i}\subset\Sc$ such that
$$
\sup_{j\notin F_{\eps, t, i}}\frac{|\Psi_\lambda(v^n)(j;t,i)-  \Psi_\lambda(v)(j;t,i)|}{W(j)}<\eps,\quad \forall n\in \N.
$$
On the finite set $F_{\eps, t, i}$, \eqref{eq:value.pointwise.convergence} implies that, for all sufficiently large $n$,
$$
\max_{j\in F_{\eps, t, i}}\frac{|\Psi_\lambda(v^n)(j;t,i) -  \Psi_\lambda(v)(j;t,i)|}{W(j)}<\varepsilon.
$$
Thus, \eqref{eq:step3} holds.
Then by the arbitrariness of $(t,i)$, $\Psi_\lambda(v^n)\to\Psi_\lambda(v)$ in the product topology of $\Xc$, which completes the proof.
\end{proof}

\begin{proof}[Proof of Theorem~\ref{thm:existence.regular}]
{\bf Part (i):} Fix $\lambda>0$. Let $M\geq M^*_\lambda$ with the constant $M^*_\lambda$ given in Lemma \ref{lm:self-mapping}. By Lemma \ref{lm:KM.compact}, the set $\Kc(M)$ is a nonempty, compact, and convex subset of $\Xc$. Since each $\Bc_{\wt}$ is a Banach space, $\Xc$, equipped with the product topology, is a Hausdorff locally convex topological vector space.
By Lemmas \ref{lm:self-mapping} and \ref{lm:Psi.continuous},$\Psi_\lambda$ maps from $\Kc(M)$ to itself, and is continuous. Therefore, the Schauder-Tychonoff fixed-point theorem implies that there exists $v^*\in \Kc(M)$ such that $$\Psi_\lambda(v^*)=v^*.$$ Define
$
\pi^*:=\Gamma_\lambda(v^*).
$
Proposition \ref{prop:fixed-points-entropy} now implies that $\pi^*$ is a regularized equilibrium for the entropy-regularized problem and $v^*=V^{\pi^*}_\lambda$.

{\bf Part (ii):} In addition, the existence of $M^*$ in Lemma \ref{lm:self-mapping} uniformly over all $\lambda\in(0,1]$ ensures  Part (ii).
\end{proof}

\section{Existence of relaxed equilibria of the original problem}\label{sec:origin}
Now, we move on to prove the existence of a relaxed equilibrium for the original MDP\eqref{1.2}. Similarly to \eqref{eq:def.V}, for any $\pi \in \Pi$, we introduce the auxiliary value function $V^\pi:\Sc\times \N_0\times \Sc$ by
\begin{equation}\label{eq:def.Vorigin}
  V^\pi(j; t, i) := \mathbb{E} \left[ \sum_{k=t+1}^{\infty} f^{\pi(k,X^\pi_k)}(k, X_k^\pi; t, i)  \bigg| X^\pi_{t+1}=j\right] 
\end{equation}
We also commonly write $V^\pi(\cdot; t, i)$ as a vector in $\R^{\infty}$. As $A$ is compact, for any $(t, y, i) \in N_0\times \Bc_{\wt} \times \Sc$,
\begin{equation}\label{4.3}
  E(t,y,i) := \argmax_{a \in A} H(t,i,y,a) \subseteq A 
\end{equation}
is nonempty and closed. For any $v\in \Xc$, consider the collection
\begin{equation}\label{eq:def.Gamma}
\Gamma(v):=\{ \pi \in \Pi : \supp(\pi(t,i)) \subseteq E(t, v(\cdot; t,i), i), \quad \forall (t,i) \in \N_0\times\Sc\}, 
\end{equation}
where $\supp(\varpi)$ denotes the support of $\varpi \in \Pc(A)$. We can then define a set-valued operator $\Psi : \Xc \to 2^{\Xc}$ by
\begin{equation}\label{eq:ef.Psi}
  \Psi(v) := \{ V^\pi : \pi \in \Gamma(v) \}. 
\end{equation}
Moreover, we can also define a set-valued operator $\Phi : \Pi \to 2^{\Pi}$ by
\begin{equation}\label{4.6}
  \Phi(\pi) := \Gamma(V^\pi) \subseteq \Pi. 
\end{equation}

\begin{Proposition}\label{prop:fixedpoint.origin}
  Assume that Assumptions~\ref{assum:reward-bound}, \ref{assum:lyapunov.p} and~\ref{assum:lipschitz} hold. Then,
  \begin{equation}\label{3.37}
 \pi \in \Pi \text{ is a relaxed equilibrium } \iff \pi \in \Phi(\pi). 
  \end{equation}
  Moreover, for any $v \in \Xc$,
  \begin{equation*}
 v = V^\pi \text{ for some relaxed equilibrium } \pi \in \Pi \iff v \in \Psi(v).
  \end{equation*}
\end{Proposition}

\begin{proof}
The proof follows a similar argument  in the proof of Proposition \ref{prop:fixed-points-entropy} (while ignoring the term $\lambda \Hc(\cdot)$ therein). Fix $\pi\in\Pi$ and $(t,i)\in\N_0\times\Sc$. For an arbitrary $\varpi\in\Pc(A)$, let $\varpi\otimes_1\pi$ denote the policy that uses $\varpi$ at the node $(t,i)$ and follows $\pi$ from time $t+1$ onward. By the definition of $V^\pi$ in \eqref{eq:def.Vorigin} and the Markov property,
$$
\begin{aligned}
 J^{\varpi\otimes_1\pi}(t,i)&= \int_A\left[f(t,i,a;t,i) + \sum_{j\in\Sc}p(t,i,j,a)V^\pi(j;t,i) \right]\varpi(da).
\end{aligned}
$$
Consequently, $\pi$ is a relaxed equilibrium if and only if, for every $(t,i)\in\N_0\times\Sc$,
$$
\pi(t,i)\in\argmax_{\varpi\in\Pc(A)}\int_A \left[f(t,i,a;t,i) + \sum_{j\in\Sc}p(t,i,j,a)V^\pi(j;t,i)\right]\varpi(da).
$$
Since the objective above is linear in $\varpi$, this is equivalent to
$$
\supp(\pi(t,i))  \subset  E\big(t,V^\pi(\cdot;t,i),i\big),  \qquad  \forall (t,i)\in\N_0\times\Sc.
$$
In view of \eqref{eq:def.Gamma} and \eqref{4.6}, this is precisely $\pi\in\Gamma(V^\pi)=\Phi(\pi)$, so \eqref{3.37} holds.

Now let $v\in\Xc$. If $v=V^\pi$ for some relaxed equilibrium $\pi\in\Pi$, then the first part gives
$
\pi\in\Gamma(V^\pi)=\Gamma(v),
$
and hence, by \eqref{eq:ef.Psi}, 
$
v=V^\pi\in\Psi(v).
$
Conversely, if $v\in\Psi(v)$, then there exists $\pi\in\Gamma(v)$ such that $v=V^\pi$. Therefore,
$
\pi\in\Gamma(v)=\Gamma(V^\pi)=\Phi(\pi),
$
and the first part implies that $\pi$ is a relaxed equilibrium.
\end{proof}
  
\begin{Theorem}\label{thm:relax.origin}
Assume that Assumptions~\ref{assum:reward-bound}--\ref{assum:cone} hold. The set-valued mapping $\Psi$ admits a fixed point. Consequently, the original problem admits a relaxed equilibrium.
\end{Theorem}

\subsection{Proof of Theorem \ref{thm:relax.origin}}

Let $\{\lambda_n\}_{n\in\mathbb N}\subset(0,1]$ satisfy
$\lambda_n\to0+$. For each $n$, Theorem \ref{thm:existence.regular} yields a fixed point $v^n\in\Kc(M^*)$ of $\Psi_{\lambda_n}$ such that
$
\pi^n:=\Gamma_{\lambda_n}(v^n)
$
is a regularized equilibrium for the entropy-regularized problem with entropy weight $\lambda_n$. In particular,
$
v^n=V_{\lambda_n}^{\pi^n}.
$

\begin{Lemma}\label{lm:zero.entropylimit}
\bi
\item[(i)] There exists a subsequence $(v^{n_k}, \pi^{n_k})_{k\in \N}$ and a pair $(v^*, \pi^*)\in \Kc(M^*)\times \Pi$ such that 
$$ v^{n_k}\to  v^*  \;\text{in }\Kc(M^*);\; \text{ and }
\pi^{n_k}(t,i)\Longrightarrow\pi^*(t,i)\;\text{ weakly in }\Pc(A),\quad  \forall (t,i)\in\N_0\times\Sc.
$$
\item[(ii)] For each $(t,i)\in\N_0\times\Sc$, $\supp(\pi^*(t,i))\subset E(t, v^*(\cdot; t, i), i)$.  Equivalently, $\pi^*\in\Gamma(v^*)$.
\ei
\end{Lemma}

\begin{proof}
  For simplicity of notation, after passing to subsequences we continue
  to write $n$ in place of $n_k$.
  
  \medskip
  \noindent
  \textbf{Part (i).}
By Lemma \ref{lm:KM.compact}, $\Kc(M^*)$ is compact under the product topology, there exist a subsequence, still denoted by $(v^n)_{n}$, and $v^*\in\Kc(M^*)$ such that 
$
v^n\to v^*\text{ in }\Kc(M^*).
$
Equivalently,
$$
\|v^n(\cdot;t,i)-v^*(\cdot;t,i)\|_{\wt}\to0,\quad \forall (t,i)\in\N_0\times\Sc.
$$  
The set $\N_0\times\Sc$ is countable; enumerate it as 
$$
\N_0\times\Sc=\{(t_1,i_1),(t_2,i_2),\ldots\}.
$$
Since $A$ is compact, $\Pc(A)$ is compact under weak convergence. Starting from the above subsequence, successively extract a subsequence along which $\pi^n(t_1,i_1)$ converges weakly, then a further subsequence along which $\pi^n(t_2,i_2)$ converges weakly, and so on. Then the diagonal subsequence, denoted by $\{\pi^{n_k}\}$, satisfies
$$
\pi^{n_k}(t,i)\Longrightarrow\pi^*(t,i), \quad \forall (t,i)\in\N_0\times\Sc,
$$
for some $\pi^*(t,i)\in\Pc(A)$. Thus $\pi^*\in\Pi$. Meanwhile, the subsequence $(v^{n_k})_k$ also converges to $v^*$, which completes the proof for Part (i).
  
\medskip
\noindent
\textbf{Part (ii).}
Fix $(t,i)\in\N_0\times\Sc$ and Recall \eqref{eq:def.H}.
By a similar argument as from \eqref{eq:product.conv.coordinate} to \eqref{eq:F.uniform.convergence}, we get
$\|H(t,i, v^{n_k}(\cdot; t, i), a)-H(t,i, v^*(\cdot; t, i), a)\|_{L^\infty(a)}\to 0.$
Let
$$
E_*:=E\big(t,v^*(\cdot;t,i),i\big) =\argmax_{a\in A}H(t,i, v^*(\cdot; t, i), a),\quad \overline H_*:=\max_{a\in A}H(t,i, v^*(\cdot; t, i), a).
$$
To prove that $\supp(\pi^*(t,i))\subseteq E_*$, it suffices to show that $\pi^*(t,i)(D)=0$ for every compact set $D\subseteq A\setminus E_*$. Fix such a set $D$. Since $D$ and $E_*$ are disjoint compact sets and $G_*(\cdot)$ is continuous, 
$$
\eps_0:=\overline H_*-\max_{a\in D}H(t,i, v^*(\cdot; t, i), a)>0.
$$
Choose $a^*\in E_*$. By define $L_{t, i}:= L_f w(i)+L_p M^*(w(i)+C_w)$, a similar argument to \eqref{eq:entropyest.1} yields
\be\label{eq:lm.limit0} 
|H(t,i, v^{n_k}(\cdot; t, i), a)-H(t,i, v^{n_k}(\cdot; t, i), a')|  \leq L_{t,i}|a-a'|,  \;\quad \forall a,a'\in A,\quad \forall k\in\N.
\ee 
Recall $\vartheta$ in Assumption \ref{assum:cone}. Set
$$
r_0:=\min\left\{\vartheta,\frac{\eps_0}{8L_{t,i}}\right\},\;  D_0:=A\cap B_{r_0}(a^*).
$$
The uniform cone condition in Assumption \ref{assum:cone} implies that $\Leb(D_0)>0$. Since
$H(t,i, v^{n_k}(\cdot; t, i), a)\to H(t,i, v^*(\cdot; t, i), a)$ on $A$ uniformly, for all sufficiently large $k$,
$$
\sup_{a\in D}H(t,i, v^{n_k}(\cdot; t, i), a) \leq \overline H_*-\frac{3\eps_0}{4}, \quad H(t,i, v^{n_k}(\cdot; t, i), a^*)\geq \overline H_*-\frac{\eps_0}{8}.
$$
For every $a\in D_0$, \eqref{eq:lm.limit0} then gives
$$
H(t,i, v^{n_k}(\cdot; t, i), a) \geq H(t,i, v^{n_k}(\cdot; t, i), a^*)-r_0L_{t,i} \geq \overline H_*-\frac{\eps_0}{4}.
$$
Therefore, as $k\to\infty$,
$$
\begin{aligned}
\pi^{n_k}(t,i)(D)=& \int_D \pi^{n_k}(a)da=\frac{\int_D\exp(H(t,i, v^{n_k}(\cdot; t, i), a)/\lambda_{n_k})\,da} {\int_A\exp(H(t,i, v^{n_k}(\cdot; t, i), a)/\lambda_{n_k})\,da}\\
&\leq  \frac{\Leb(D)\exp\big((\overline H_*-3\varepsilon/4)/\lambda_{n_k}\big)}{\Leb(D_0)  \exp\big((\overline H_*-\varepsilon/4)/\lambda_{n_k}\big)}  
=  \frac{\Leb(D)}{\Leb(D_0)}  \exp\left(-\frac{\varepsilon}{2\lambda_{n_k}}\right)  \to 0.
\end{aligned}
$$
Since $D$ is closed and $\pi^n(t,i)$ converges to $\pi^*(t,i)$ weakly in $\Pc(A)$,
we then conclude from above that $\pi^*(t,i)(D)=0$. This proves  Part (ii).
\end{proof}  

\begin{Lemma}\label{lm:zero.entropylimitV}
Under the notations and assumptions in Lemma \ref{lm:zero.entropylimit}, for every $(t,i,j)\in\N_0\times\Sc^2$,
\be\label{eq:lm.limit1}
V^{\pi^{n_k}}(j;t,i) \to V^{\pi^*}(j;t,i).
\ee
Moreover,
$
v^*=V^{\pi^*}.
$
\end{Lemma}

\begin{proof}
The proof for \eqref{eq:lm.limit1} follows a similar argument as in the proof of Lemma \ref{lm:Psi.continuous}. And fix $(t,i,j)\in\N_0\times\Sc^2$, we define the following similar notations: 
For $N\in\N$, define
$$
\begin{aligned}
&V_N^{\pi^{n_k}}(j;t,i):=\E\left[\sum_{m=1}^{N}f^{\pi^{n_k}(t+m,X_{t+m}^{\pi^{n_k}})}(t+m,X_{t+m}^{\pi^{n_k}};t,i)\,\middle|\,X_{t+1}^{\pi^{n_k}}=j\right],\\
&V_N^{\pi^*}(j;t,i):=\E\left[\sum_{m=1}^{N}f^{\pi^*(t+m,X_{t+m}^{\pi^*})}(t+m,X_{t+m}^{\pi^*};t,i)\,\middle|\,X_{t+1}^{\pi^*}=j\right].
\end{aligned}
$$
Then by ignoring the entropy part, a similar argument from Step 1 to Step 3 in the proof of Lemma \ref{lm:Psi.continuous} yields 
$$
\begin{aligned}
  &\E\left[  f^{\pi^{n_k}(t+m,X_{t+m}^{\pi^{n_k}})}  (t+m,X_{t+m}^{\pi^{n_k}};t,i)  \,\middle|\,  X_{t+1}^{\pi^{n_k}}=j  \right]  \\
  &\to   \E\left[  f^{\pi^*(t+m,X_{t+m}^{\pi^*})}  (t+m,X_{t+m}^{\pi^*};t,i)  \,\middle|\,  X_{t+1}^{\pi^*}=j  \right].
\end{aligned}
$$
This further implies that $V_N^{\pi^{n_k}}(j;t,i)\to V_N^{\pi^{*}}(j;t,i)$ for all $(t,i)\in \N_0\times\Sc$. Then a similar argument in Step 4 therein (by ignoring the entropy term again) shows
\be\label{eq:lmlimit1.v} 
V^{\pi^{n_k}}(j;t,i)\to V^{\pi^{*}}(j;t,i),\quad \forall (t,i, j)\in \N_0\times\Sc^2.
\ee 

It remains to identify the limit $v^*$. Since
$v^{n_k}=V_{\lambda_{n_k}}^{\pi^{n_k}}$,
$$
v^{n_k}(j;t,i)-V^{\pi^{n_k}}(j;t,i)=\lambda_{n_k}\E\left[  \sum_{m=1}^{\infty}  \widetilde\delta(m)\Hc(\pi^{n_k}(t+m,X_{t+m}^{\pi^{n_k}}))  \,\middle|\,  X_{t+1}^{\pi^{n_k}}=j  \right].
$$
By Lemma \ref{lm:entropy.est}, the fact that $v^{n_k}\in\Kc(M^*)$, and $\lambda_{n_k}\in(0,1]$, there exists a constant $C>0$, independent of $n,m,r$, such that
$$
|\Hc(\pi^{n_k}(t+m,r))|  \leq  C\left(  1+|\ln\lambda_{n_k}|+w(r)  \right).
$$
We then obtain from Lemma \ref{lm:est.0} that
$$
\begin{aligned}
|v^{n_k}(j;t,i)-V^{\pi^{n_k}}(j;t,i)|  &\leq C\lambda_{n_k}  \sum_{m=1}^{\infty}\widetilde\delta(m)  \left(  1+|\ln\lambda_{n_k}|  +\gamma_w^{m-1}w(j)  +\frac{b_w}{1-\gamma_w} \right).
\end{aligned}
$$
Then by
$
\lambda_{n_k}\to 0,$
$\lambda_{n_k}|\ln\lambda_{n_k}|\to 0 $
and 
$\sum_{m=1}^{\infty}\widetilde\delta(m)<\infty$,  the right-hand side of the above inequality converges to zero. Therefore,
$$
v^{n_k}(j;t,i)-V^{\pi^{n_k}}(j;t,i)\to 0.
$$
By Part (i) in Lemma \ref{lm:zero.entropylimit}, $v^{n_k}(j;t,i)\to  v^*(j;t,i)$ for all $(t,i,j)\in \N_0\times \Sc^2$. This together with \eqref{eq:lmlimit1.v} tells 
$$
v^*(j;t,i)=V^{\pi^*}(j;t,i),\quad\forall(t,i,j)\in\N_0\times\Sc^2.
$$
\end{proof}

\begin{proof}[Proof of Theorem \ref{thm:relax.origin}]
  Lemma \ref{lm:zero.entropylimit} Part (ii) and Lemma \ref{lm:zero.entropylimitV} together tells that 
  $$\supp(\pi^*(t,i))\subset E(t, v^*(\cdot; t, i), i)=E(t, V^{\pi^*}(\cdot; t, i), i),\quad \forall (t,i)\in \N_0\times \Sc.$$
  Then by Proposition \ref{prop:fixedpoint.origin}, $\pi^*$ is a relaxed  equilibrium for the original problem.
\end{proof}

\section{Policy iteration under weighted discounting and its regret analysis}\label{sec:pia}

Throughout this section, we study the policy iteration algorithm. We first show the exponential convergence of PIA for sufficiently small discounting factors. To complement the convergence result, we also provide a counterexample in which PIA fails to converge when discounting is weak. Finally, we provide a quantitative regret analysis showing that PIA generates approximate equilibria.

 In contrast to the fixed-point argument in Sections \ref{sec:entropy}--\ref{sec:origin}, exponential convergence of PIA  requires a global norm that is uniform in the calendar time and the evaluating state. We therefore introduce
\be\label{eq:PIA.global.norm}
\|v\|_{\widetilde \Xc}  :=  \sup_{(t,i)\in\N_0\times\Sc}  \frac{\|v(\cdot;t,i)\|_{\wt}}{w(i)},  \quad  v\in\Xc,
\ee
and set
$$
\widetilde \Xc :=\left\{ v\in\Xc:\ \|v\|_{\widetilde \Xc}<\infty \right\}.
$$
Since $w/W$ is bounded and $w\geq 1$, $\Kc(M)\subset\widetilde \Xc$ for every $M<\infty$.

Fix $\lambda>0$. Starting from $v^0\in\Kc(M_\lambda^*)$, we introduce the \textit{policy iteration algorithm} (PIA) as follows:
\begin{itemize}
  \item \textbf{Initialization:} Start with an initial guess $v^0\in \Kc(M^*_\lambda)$.
  \item \textbf{Iteration Step $n\geq 0$:} 
  \begin{enumerate}
 \item \textit{Policy Update:} Given $v^{n}$, the updated policy $\pi^{n+1}$ is computed as:
 \bee
   \pi^{n+1}(t,i, a) := \Gamma_\lambda(v^n)(t,i, a)=\frac{e^{\frac{1}{\lambda} [f(t,i,a;t,i)+p(t,i,a)\cdot v^n(\cdot; t, i)]}}{\int_A e^{\frac{1}{\lambda} [f(t,i, a';t,i)+p(t,i,a')\cdot v^n(\cdot; t, i)]} da'}.
 \eee
 \item {\it Policy Evaluation:} Compute the updated auxiliary value function 
 $$v^{n+1}(j; t,i):=V^{\pi^{n+1}}(j; t, i),\quad (t,i,j)\in \N_0\times \Sc^2.$$
 Equivalently,
 $
 v^{n+1}=\Psi_\lambda(v^n).
 $
  \end{enumerate}
\end{itemize}
Notice that, by Lemma~\ref{lm:self-mapping}, $(v^n)_{n\ge0}$ remains in $\Kc(M_\lambda^*)$ whenever $v^0\in\Kc(M_\lambda^*)$.

We first show the convergence of PIA with extra assumptions. 
\begin{Assumption}\label{assum:weight.discounting}
  The following conditions hold.
\begin{enumerate}
\item We further require $\gamma_w=\gamma_{W}=0$ in Assumption~\ref{assum:lyapunov.p}. Thus, uniformly in $(t,i,a)\in\N_0\times\Sc\times A$,
\be \label{eq:PIA.zero.drift}
\sum_{j\in\Sc}p(t,i,j,a)w(j)\leq b_w, \quad \sum_{j\in\Sc}p(t,i,j,a)W(j)\leq b_{W}.
\ee
\item The reward has the following weighted-discounting representation:
\be\label{eq:PIA.weighted.reward}
f(s,j,a;t,i) = \widetilde\delta(s-t)g(j,a;t,i),  \qquad 0\leq t\leq s,
\ee
where
\be \label{eq:PIA.weighted.discount}
\widetilde\delta(k) =  \int_{[0,\rho_0]}\rho^k\,F(d\rho),\quad k\in\N_0,
\ee
for a probability measure $F$ on $[0,\rho_0]$, with $\rho_0\in(0,1]$. The entropy term is discounted by the same function $\widetilde\delta$.
 
\item There exists positive constants $K_g, K_H<\infty$ such that
\begin{align}
&\operatorname{osc}_{a\in A}g(j,a;t,i)  \leq K_g, \quad  \forall(t,i, j)\in\N_0\times\Sc^2,  \label{eq:PIA.g.osc}\\
&\operatorname{osc}_{a\in A}\left[g(i,a;t,i)+p(t,i,a)\cdot v(\cdot;t,i)\right]\leq K_H,\quad\forall(t,i)\in\N_0\times\Sc, \;\forall v\in\Kc(M_\lambda^*) \label{eq:PIA.Hamiltonian.osc}
\end{align}
 
\end{enumerate}
\end{Assumption}

\begin{Remark}
Condition \eqref{eq:PIA.Hamiltonian.osc} is a uniform oscillation condition on the full one-step Hamiltonian. It is needed to obtain a global entropy-difference estimate that closes in the norm \eqref{eq:PIA.global.norm}. It is satisfied, for example, if \eqref{eq:PIA.g.osc} holds and the action dependence of $p$ satisfies a uniform weighted bound strong enough that
$$
\sup_{t,i}\sup_{a,a'\in A}\sum_{j\in\Sc}  |p(t,i,j,a)-p(t,i,j,a')|W(j)<\infty
$$
after restriction to $\Kc(M_\lambda^*)$. And $K_H$ is taken independently over all  $\lambda(0, 1]$ by the uniform choice of $M^*$.
\end{Remark}

Given a policy $\pi\in\Pi^r$ and $\rho\in[0,\rho_0]$, we define for $(t,s, i)\in \Delta\times \Sc$ that
\be\label{eq:PIA.rho.value}
\begin{aligned}
&U_\rho^\pi(s,j;t,i)  :=\\  &\E\Bigg[\sum_{m=0}^{\infty}\rho^{m+1}\Big(&g^{\pi(s+m,X_{s+m}^\pi)}(X_{s+m}^\pi;t,i)+\lambda\Hc(\pi(s+m,X_{s+m}^\pi))\Big)  \,\Bigm|\,  X_s^\pi=j\Bigg].
\end{aligned}
\ee
Then the weighted-discounting representation gives
\be\label{eq:PIA.mixture.identity}
V_\lambda^\pi(j;t,i)  =  \int_{[0,\rho_0]}U_\rho^\pi(t+1,j;t,i)\,F(d\rho).
\ee
Moreover, the Markov property yields that
\be\label{eq:vn.rho}
\begin{aligned}
U_\rho^\pi(s,j;t,i)=\rho\Bigg[&g^{\pi(s,j)}(j;t,i)+\lambda\Hc(\pi(s,j)) +\sum_{r\in\Sc}p^\pi(s,j,r)U_\rho^\pi(s+1,r;t,i)\Bigg].
\end{aligned}
\ee
Similar to \eqref{eq:PIA.global.norm}, for a family $U=(U(s,\cdot;t,i))_{0\leq t<s,\ i\in\Sc}$, define
\be\label{eq:PIA.U.norm}
\|U\|_{\widehat \Xc}  :=  \sup_{\substack{t\in\N_0,\ i\in\Sc\\s\geq t+1}}  \frac{\|U(s,\cdot;t,i)\|_{\wt}}{w(i)}.
\end{equation}

\begin{Theorem}\label{thm:converge.PI}
Suppose Assumptions~\ref{assum:reward-bound}--\ref{assum:cone} and Assumption~\ref{assum:weight.discounting} hold, and fix $\lambda>0$. There exist finite constants $C_0, C_1>0$, independent of $\rho_0$, such that, if
\be\label{eq:PIA.small.rho}
\rho_0C_0<1\quad\text{and}\quad \theta_\lambda(\rho_0) := \frac{  \rho_0C_1\lambda^{-1}}{ 1-\rho_0C_0}<1,
\ee
then the PIA converges exponentially in $\widetilde \Xc$ by
\be\label{eq:PIA.exponential}
\|v^n-v^*\|_{\widetilde \Xc} \leq \theta_\lambda(\rho_0)^n \|v^0-v^*\|_{\widetilde \Xc}, \quad n\ge0.
\ee
In addition, $C_0, C_1$ can be chosen uniformly over $\lambda(0, 1]$.
\end{Theorem}

\begin{Corollary}\label{cor:PIA.unique}
  Under the assumptions of Theorem~\ref{thm:converge.PI}, if
  $\theta_\lambda(\rho_0)<1$, then $\Psi_\lambda$ has at most one fixed
  point in $\widetilde \Xc$. Consequently, the entropy-regularized
  equilibrium auxiliary value is unique in $\widetilde \Xc$, and the
  corresponding regularized equilibrium is unique.
\end{Corollary}

\begin{proof}
  If $v^*$ and $\widetilde v^*$ are two fixed points, the same estimate
  as in the proof of Theorem~\ref{thm:converge.PI} gives
  $$
  \|v^*-\widetilde v^*\|_{\widetilde \Xc}
  \le
  \theta_\lambda(\rho_0)
  \|v^*-\widetilde v^*\|_{\widetilde \Xc}.
  $$
  Since $\theta_\lambda(\rho_0)<1$, the two fixed points coincide.
  The Gibbs selector is single-valued, so the corresponding equilibrium
  policies also coincide.
\end{proof}

\begin{Definition}\label{def:weighted.epsilon.equilibrium}
 A policy $\pi\in\Pi$ is called a weighted $\eps$-equilibrium for the original problem if
$$\sup_{\varpi\in\Pc(A)}  J^{\varpi\otimes_1\pi}(t,i) \leq  J^\pi(t,i)+\eps w(i),  \quad  \forall(t,i)\in\N_0\times\Sc.
 $$
\end{Definition}

\begin{Corollary}\label{cor:PIA.epsilon}
 Assume that, for some finite $L_{H}>0$,
\begin{equation}\label{eq:cor.osc}
 \sup_{\substack{(t,i)\in\N_0\times\Sc\\y\in\Kc(M^*)}} \operatorname{Lip}_{a\in A} \left[ g(i,a;t,i)+p(t,i,a)\cdot y \right] \leq L_{H}.
\end{equation}
Suppose the assumptions of Theorem \ref{thm:converge.PI} hold. Then the $n$-th PIA policy  $\pi^n$ is a weighted $\eps_n$-equilibrium for the original  problem with
  \begin{equation}\label{eq:PIA.epsequi}
 \eps_n \leq C\left[\lambda(1+|\ln\lambda|) +  \left( \frac{  \rho_0C_1\lambda^{-1}}{ 1-\rho_0C_0} \right)^{n-1} \|v^0-v^*\|_{\widetilde\Xc}\right],
  \end{equation}
where $C_0, C_1$ are finite constants in Theorem \ref{thm:converge.PI}, and $C>0$ is another finite constant independent of $\lambda\in(0,1]$ and $n$. 

\end{Corollary}

\begin{Remark}
The weight $w(i)$ in the last line of the proof comes from the error estimate in Theorem \ref{thm:converge.PI} for PIA, and hence, if the convergence of PIA in \eqref{eq:PIA.exponential} is achieved, instead of in terms of weighted norm $\|\cdot\|_\Xc$, in uniform bounded norm. Then the weighted $\eps$-equilibrium in Corollary \ref{cor:PIA.epsilon} becomes an ordinary state-uniform $\eps$-equilibrium result, and a simple condition makes it true is: $\sup_{i\in \Sc}w(i)<\infty$.
\end{Remark}

\begin{Remark}\label{rm:uniform.H.Lipschitz}
 Condition \eqref{eq:cor.osc} is satisfied in several natural special cases, such as the state space being finite, or the transition function being independent of the action.
More generally, suppose
$
 \sup_{t,i}  \operatorname{Lip}_{a\in A}g(i,a;t,i)  \leq L_g<\infty
$
 and there exists $L_p^0<\infty$ such that
 $$
 \sum_{j\in\Sc}  |p(t,i,j,a)-p(t,i,j,a')|W(j) \leq\frac{L_p^0}{w(i)+C_w}|a-a'|, \quad \forall(t,i),\ a,a'\in A.
  $$
  Then, by \eqref{eq: assum.wW0} and \eqref{eq:def.KM}, for any $y\in\Kc(M_\lambda^*)$,
  $
  \|y(\cdot;t,i)\|_{\wt}  \leq  M_\lambda^*(w(i)+C_w).
  $
Hence,
  $$
  \operatorname{Lip}_{a\in A}  H(t,i,y(\cdot;t,i),a)  \leq  L_g+L_p^0M_\lambda^*, \quad \forall (t,i)\in \N_0\times \Sc, \forall y\in\Kc(M_\lambda^*). 
  $$
 Thus
  \eqref{eq:cor.osc} holds.
\end{Remark}

\subsection{A numerical illustration of PIA convergence}

We now provide a numerical example illustrating the convergence of PIA in the strongly discounted regime of Theorem~\ref{thm:converge.PI}. Consider the finite state space $\Sc=\{1,\ldots,6\}$
and the action space
$
A=\{a=(a_1,a_2)\in\R^2:a_1^2+a_2^2\leq1\}.
$
Since $\Sc$ is finite, we take the weight functions $w$ and $W$ to be constant, so that
the norm $\|\cdot\|_{\widetilde\Xc}$ in \eqref{eq:PIA.global.norm} is equivalent to the usual supremum norm.

For $i,j\in S$ and $a\in A$, let
$$
\begin{cases}
p(i,j,a)=\frac{\exp\{-(j-i-a_1)^2+a_2{\bm 1}_{\{j=i\}}\}}
{\sum_{k\in S}\exp\{-(k-i-a_1)^2+a_2{\bm 1}_{\{k=i\}}\}},\\
g(i,a)=R_{\rm base}(i)+100ia_1-(a_1^2+a_2^2) 
\end{cases}
$$
where
$R_{\rm base}(1)=R_{\rm base}(6)=-10^4$, and $R_{\rm base}(i)=0$ for $2\leq i\leq5$.
We take $\lambda=2.5$ and the weighted discounting
$$
\widetilde \delta(k)=\frac12(0.85)^k+\frac12(0.10)^k,
$$
which is of the form \eqref{eq:PIA.weighted.discount} with
$
F=\frac12\delta_{0.85}+\frac12\delta_{0.10}.
$

Starting from $v^0=0$, we implement the PIA: $v^{n+1}=\Psi_\lambda(v^n)$ and
$\pi^{n+1}=\Gamma_\lambda(v^n)$. Since the model is time-homogeneous, both $\pi^{n+1}$ and $v^{n+1}$ depend only on the state variable. Denote by $p^{\pi^{n+1}}$ the transition matrix under $\pi^{n+1}$, 
and write
$r^{n+1}(j):=g^{\pi^{n+1}(j)}(j)+\lambda\Hc(\pi^{n+1}(j))$ for $j\in S.$
Then, the policy evaluation step becomes
$$
v^{n+1}=\left[\frac{0.85}{2}\big(I-0.85p^{\pi^{n+1}}\big)^{-1}+\frac{0.10}{2}\big(I-0.10p^{\pi^{n+1}}\big)^{-1}\right]r^{n+1}.
$$
Figure~\ref{fig:convergence} reports the iteration error
$
\|v^n-v^{n-1}\|_\infty.
$
Its approximately linear decay on the logarithmic scale exhibits the exponential convergence of PIA in this strongly discounted example, in accordance with the convergence result of Theorem~\ref{thm:converge.PI}.
\begin{figure}[htbp]
	\centering
\includegraphics[width=0.85\textwidth]{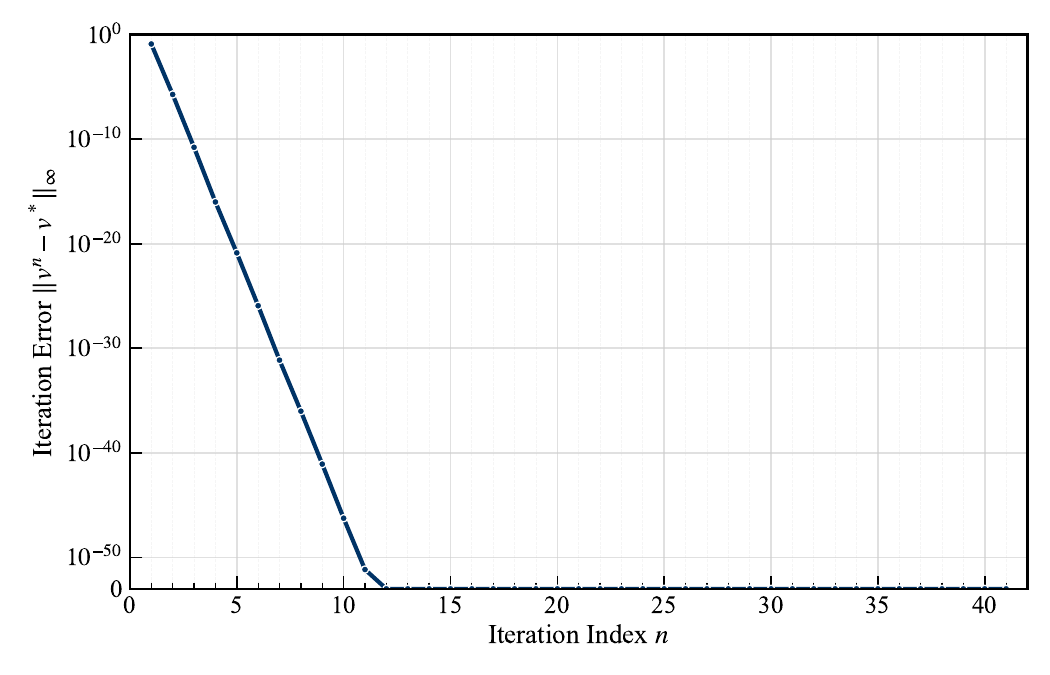}
\caption{Iteration error $\|v^n-v^{n-1}\|_\infty$ for PIA with
	$\lambda=2.5$ and
	$\tilde\delta(k)=\frac12(0.85)^k+\frac12(0.10)^k$.}
	\label{fig:convergence}
\end{figure}

\subsection{An example for the failure of PIA convergence under weak discountings}
\label{subsec:PIA.nonconvergent.interval}

In this subsection, we provide a simple example showing that, when the discounting is sufficiently weak, the iteration
$v^{n+1}=\Psi_\lambda(v^n)$ may fail to converge even when the action space is a compact interval.
The example is time-homogeneous and the reward is of the weighted-discounting form
\eqref{eq:PIA.weighted.reward}, so it is a special case of the framework considered above.

Take the finite state space $\Sc=\{1,2\}$ and the action space $A=[0,1]$. The transition function is
time-homogeneous and is given by
$$
(p(i,j,a))_{i,j\in \Sc}=
\left[
\begin{matrix}
 1&0\\
 a&1-a
\end{matrix}
\right]_{ij},\qquad a\in[0,1].
$$
Thus state $1$ is absorbing, while at state $2$ the action $a$ is the probability of moving to
state $1$. Let
$$
g(1,a)=0,\qquad g(2,a)=-2-4a,
$$
and set
$$
f(s,j,a;t,i)=\widetilde\delta(s-t)g(j,a),\qquad
\widetilde\delta(k)=\frac12(0.99)^k+\frac12(0.01)^k.
$$
The entropy term is discounted by the same $\widetilde\delta$, and we take $\lambda=0.01$.
Notice that \eqref{eq:PIA.weighted.discount} holds with
$F=\frac12\delta_{0.99}+\frac12\delta_{0.01}$.

Since the model is time-homogeneous and $g$ is independent of the evaluating pair $(t,i)$,
the subspace of functions of the form
$$
v(j;t,i)=y(j),\qquad y=(y(1),y(2))^\top\in\R^2,
$$
is invariant under $\Psi_\lambda$. On this subspace, \eqref{2.5} and \eqref{2.8} reduce to a
two-dimensional iteration, which we denote again by $y\mapsto\Psi_\lambda(y)$.

\begin{Proposition}\label{prop:PIA.two.cycle.interval}
 For the above model, the iteration $y^{n+1}=\Psi_\lambda(y^n)$ with $y^0=(0,0)^\top$ does not
 converge. More precisely, for $n\geq1$ one has $y^n(1)=0$, and if $x_n:=y^n(2)$, then
 $x_{n+1}=G(x_n)$ for a continuously differentiable function $G: \R\to\R$ satisfying
 $$
 G(J_-)\subset J_+,\qquad G(J_+)\subset J_-,
 $$
 where
 $$
 J_-=[-51,-50],\qquad J_+=[-3.04,-3.03].
 $$
 Moreover,
 $$
 \sup_{x\in J_-\cup J_+}|(G^2)'(x)|<0.003.
 $$
 Consequently, $G$ has an attracting period-two orbit
 $\{x_-^*,x_+^*\}$ with $x_-^*\in J_-$ and $x_+^*\in J_+$, and the iteration starting from
 $x_0=0$ converges to this two-cycle rather than to a fixed point.
\end{Proposition}

\begin{proof}
 At state $1$, both the transition and the reward are independent of the action. Hence
 $\Gamma_\lambda(y)(1,a)\equiv1$ on $[0,1]$, its entropy value equals to zero, and state $1$ remains absorbing. It follows that $\Psi_\lambda(y)(1)=0$ for every $y\in\R^2$. Therefore,
 $y^n(1)=0$ for every $n\geq1$.
 
Now fix $y=(0,x)^\top$, and we shall rewrite the PIA in terms of $x$. At state $2$, the one-step Hamiltonian appearing in \eqref{2.5} becomes
$$
g(2,a)+p(2,a)\cdot y=-2-4a+(1-a)x=-2+x-(x+4)a.
$$
Therefore the Gibbs density at state $2$ is
\be\label{eq:gamma.s}
\Gamma_\lambda(y)(2,a)=\frac{e^{sa}}{\int_0^1e^{sa'}da'},\quad \text{where }s=-\frac{x+4}{\lambda}=-100(x+4).
\ee
Define
\be\label{eq:eg.m}
m(x):=\int_0^1a\Gamma_\lambda(y)(2,a)da=\frac{e^s}{e^s-1}-\frac1s,
\ee
then Taylor expansion yields $\lim_{x\to 4} m(x)= \lim_{s\to 0} \left( \frac{e^s}{e^s-1}-\frac1s \right)=1/2$. Moreover,
\be\label{eq:eg.Hm}
\Hc(\Gamma_\lambda(y)(2))=\ln\left(\frac{e^s-1}{s}\right)-s\,m(x),
\ee
where the quotient $(e^s-1)/s$ is positive for every $s\neq0$. Under this Gibbs policy $\Gamma_\lambda((0,x)^T)$, the transition matrix and the regularized one-step reward vector are
$$
\begin{cases}
P(x):=
\left[
\begin{matrix}
 1&0\\
 m(x)&1-m(x)
\end{matrix}
\right],\\
R(x):= \left[
\begin{matrix}
 g^{\Gamma_\lambda(y)(1)}(1) +\lambda \Hc(\Gamma_\lambda(y(1)))\\
 g^{\Gamma_\lambda(y)(2)}(2)++\lambda \Hc(\Gamma_\lambda(y(2)))
\end{matrix}
\right]=
\left[
\begin{matrix}
 0\\
 -2-4m(x)+0.01\Hc(\Gamma_\lambda(y)(2))
\end{matrix}
\right].
\end{cases}
 $$
Recall the definition of  the auxiliary value function in \eqref{eq:def.V} and the geometric-mixture form of $\widetilde\delta$, we have
$$
\begin{aligned}
\Psi_\lambda(y)=\E\left[  \right]
=&\frac{1}{2}\left(\sum_{k=0}^\infty (0.99)^{k+1} P^k(x) \cdot R(x) + \sum_{k=0}^\infty 0.01^{k+1} P^k(x) \cdot R(x) \right)\\
=&\left[\frac{0.99}{2}(I-0.99P(x))^{-1}+\frac{0.01}{2}(I-0.01P(x))^{-1}\right]R(x).
\end{aligned}
 $$
Since the first component of $R(x)$ is zero, the second component is 
\be\label{eq:PIA.F.interval}
G(x):=\left[\frac{0.495}{0.01+0.99m(x)}+ \frac{0.005}{0.99+0.01m(x)}\right]
 \left[-2-4m(x)+0.01\Hc(\Gamma_\lambda(y)(2))\right],
 \ee
with $s$ defined in \eqref{eq:gamma.s}. Therefore, the PIA now reduced to a 1-dimensional iteration $x_{n+1}=G(x_n)$.
 
We first show the alternating property. If $x\in J_-=[-51,-50]$, then \eqref{eq:eg.m} and \eqref{eq:eg.Hm} together give that
\be\label{eq:J-.bounds} 
0.99978260<m(x)<0.99978724,\quad -7.456<\Hc(\Gamma_\lambda(y)(2))<-7.433.
\ee 
Substituting these bounds into \eqref{eq:PIA.F.interval} gives
$
-3.03749<G(x)<-3.03738,
$
and therefore
$$
G(J_-)\subset[-3.03749,-3.03738]\subset J_+.
$$
If $x\in J_+=[-3.04,-3.03]$, then $s\in[-97,-96]$, which leads to
\be\label{eq:J+.bounds} 
0.0103092<m(x)<0.0104167, \quad-3.575<\Hc(\Gamma_\lambda(y)(2))<-3.564.
\ee
Substitution into \eqref{eq:PIA.F.interval} yields
$
-50.892<G(x)<-50.632,
$
and hence
$$
G(J_+)\subset[-50.892,-50.632]\subset J_-.
$$
 
It remains to verify that the alternating orbit is attracting. Notice that $ds/dx=-100$, then by differentiating $m, \Hc(\gamma_\lambda(y)(2))$ in \eqref{eq:eg.m} and \eqref{eq:eg.Hm} respectively, we get
$$
\begin{aligned}
&-100\frac{dm}{ds}=-100\left(\frac1{s^2}-\frac{e^s}{(e^s-1)^2}\right)=\frac{dm}{dx},\\ &\frac{d}{dx}\Hc(\Gamma_\lambda(y)(2))= -100\frac{d}{ds}\Hc(\Gamma_\lambda(y)(2))
=100s\frac{dm}{ds}.
\end{aligned}
 $$
By setting
\be\label{eq:eg.orbit0}  
M(m):=\frac{0.495}{0.01+0.99m}+\frac{0.005}{0.99+0.01m} \;\text{ and }\;
R_2(x):=-2-4m(x)+0.01\Hc(\Gamma_\lambda(y)(2)),
\ee
we get that
\be\label{eq:eg.orbit1} 
M'(m)=-\frac{0.495(0.99)}{(0.01+0.99m)^2}-\frac{0.005(0.01)}{(0.99+0.01m)^2},
\quad
\frac{d}{dx}R_2=400 \frac{dm}{ds}+s\frac{dm}{ds}=-100x\frac{dm}{ds}.
\ee 
Then differentiating \eqref{eq:PIA.F.interval} gives
\be\label{eq:eg.orbit} 
G'(x)=-100\frac{dM}{dm}\frac{dm}{ds}R_2(x)+M(m(x))\frac{d}{dx}R_2,
\ee 
{\bf (i).} For $x\in J_-$, $s\in[4600,4700]$ and hence $0<dm/ds<1/4600^2$. By plugging this with the bounds in \eqref{eq:J-.bounds} into \eqref{eq:eg.orbit0} and \eqref{eq:eg.orbit1} then applying \eqref{eq:eg.orbit}, we conclude that
$$
\sup_{x\in J_-}|G'(x)|<1.1\times10^{-4}.
$$
{\bf (ii).} For $x\in J_+$, $s\in[-97,-96]$ and hence $1.0628\times 10^{-4}<\frac{dm}{ds}<1.0851\times 10^{-4}.$ Then the  substitution into the same formulas gives
$$
\sup_{x\in J_+}|G'(x)|<26.
$$
Since $G$ interchanges $J_-$ and $J_+$,
$$
|(G^2)'(x)|=|G'(G(x))||G'(x)|<26(1.1\times10^{-4})<0.003,\quad \forall x\in J_-\cup J_+.
$$
Thus $G^2$ is a strict contraction from $J_-$ into itself and also from $J_+$ into itself. Then the Banach fixed-point theorem gives unique points $x_-^*\in J_-$ and $x_+^*\in J_+$ such that
$G^2(x_-^*)=x_-^*$ and $G^2(x_+^*)=x_+^*$. Since $G(J_-)\subset J_+$ and $G(J_+)\subset J_-$, uniqueness implies
$$
G(x_-^*)=x_+^*,\qquad G(x_+^*)=x_-^*.
$$
The intervals are disjoint, so $x_-^*\neq x_+^*$ and this is a genuine period-two orbit.
 
 Finally, the initial point used by PIA indeed enters this attracting region. We compute from
 \eqref{eq:PIA.F.interval} that
 $$
 G(0)=-81.74649\ldots,\quad G^2(0)=-3.03991\ldots\in J_+,
 $$
 and therefore $G^3(0)\in J_-$. Hence, the even and odd subsequences converge to the two
 different points of the cycle. Hence $y^n$, and equivalently the PIA sequence
 $v^{n+1}=\Psi_\lambda(v^n)$ with $v^0=(0, 0)$, does not converge.
\end{proof}

\begin{Remark}\label{rm:PIA.unique.unstable.fixed.point}
A more detailed calculus analysis of the functions introduced in the proof of Proposition	\ref{prop:PIA.two.cycle.interval} yields the following properties of $G$ defined in \eqref{eq:PIA.F.interval}:
\be\label{eq:PIA.Fprime.fixed}
	\begin{cases}
		G(x)-x >4+0.01s-3.944-0.00986s =0.056+0.00014s>0,\quad \forall x\in(\infty, -4];\\
		G'(x) =-100\frac{dm}{ds} \left[M'(m(x))R(x)+M(m(x))x \right]<0,\quad \forall x\in[-4,\infty);\\
		G(-4)>-4,\quad G(-3.99)<-3.99;\quad  G'(x) <-10,\quad \forall x\in(-4,-3.99).
	\end{cases}
	\ee
Therefore, $G$ has a unique fixed point $x^*$ with $-4<x^*<-3.99$, and $x^*\approx -3.9982167$ numerically. Hence, there exists a unique regularized equilibrium for this example Moreover, $G'(x) <-10$ on $(-4,-3.99)$ implies the unique fixed point $x^*$ is {\it unstable}. In particular, the PIA does not converge to $x^*$ from any sufficiently small punctured neighborhood of $x^*$; instead, as shown in Proposition \ref{prop:PIA.two.cycle.interval}, the iteration is attracted to the period-two orbit.

This example also highlights an important distinction between the convergence behavior of PIA in finite- and infinite-horizon time-inconsistent stochastic control problems. In the finite-horizon diffusion setting of \cite{huang2026policy}, PIA converges to the unique equilibrium. In contrast, our example shows that this  convergence property need not persist in the infinite-horizon setting: even for a time-homogeneous Markov chain with only two states and a compact one-dimensional action space, the unique equilibrium can be unstable under PIA, with the iterates instead converging to an attracting period-two orbit.
\end{Remark}

\subsection{Proofs of Theorem \ref{thm:converge.PI} and Corollary \ref{cor:PIA.epsilon}}

\begin{Lemma}\label{lm:PIA.est0}
	Let $z_1, z_2\in L^\infty(A)$.
	We have
	$$\left\| \frac{e^{z_1(a)/\lambda}}{\int_A e^{z_1(a')/\lambda}\,da'}-\frac{e^{z_2(a)/\lambda}}{\int_A e^{z_2(a')/\lambda}\,da'} \right\|_{L^1(A)}\leq\frac{2}{\lambda}\|z_1-z_2\|_{L^\infty(A)}.$$
	Consequently, 
	$$
	\left\| \frac{e^{z_1(a)/\lambda}}{\int_A e^{z_1(a')/\lambda}\,da'}-\frac{e^{z_2(a)/\lambda}}{\int_A e^{z_2(a')/\lambda}\,da'} \right\|_{\mathrm{TV}}\leq\frac{1}{\lambda}\|z_1-z_2\|_{L^\infty(A)}.
	$$
\end{Lemma}

\begin{proof}
	For $s\in[0,1]$, set $z_s:=z_2+s(z_1-z_2)$ and $\pi_s(a):=\frac{e^{z_s(a)/\lambda}}{\int_A e^{z_s(a')/\lambda}\,da'}$.
	Then $\pi_s$ is differentiable in $s$ as an $L^1(A)$-valued map. A direct calculation gives
	$$
	\partial_s\pi_s(a)=
	\frac{1}{\lambda}\pi_s(a)  \left[  z_1(a)-z_2(a)  -  \int_A(z_1(a')-z_2(a'))\pi_s(a')\,da'  \right].
	$$
	Hence
	$$
	\begin{aligned}
		\|\partial_s\pi_s\|_{L^1(A)} &\leq \frac{1}{\lambda}  \int_A \pi_s(a)  \left(  |z_1(a)-z_2(a)|  +  \int_A |z_1(a')-z_2(a')|\pi_s(a')\,da'  \right)da   \\
		&\leq  \frac{2}{\lambda}\|z_1-z_2\|_{L^\infty(A)}.
	\end{aligned}
	$$
	Therefore,
	$$
	\left\| \frac{e^{z_1(a)/\lambda}}{\int_A e^{z_1(a')/\lambda}\,da'}-\frac{e^{z_2(a)/\lambda}}{\int_A e^{z_2(a')/\lambda}\,da'} \right\|_{L^1(A)}  \leq  \int_0^1\|\partial_s\pi_s\|_{L^1(A)}\,ds  \leq  \frac{2}{\lambda}\|z_1-z_2\|_{L^\infty(A)}.
	$$
	The total variation estimate follows from
	$\|\mu-\nu\|_{\mathrm{TV}}=\frac12\|d\mu/d a-d\nu/d a\|_{L^1(A)}$ for absolutely continuous probability measures.
\end{proof}

Let $v^*$ denote the fixed point of $\Psi_\lambda$ obtained in Theorem~\ref{thm:existence.regular}, and set
$
\pi^*:=\Gamma_\lambda(v^*).
$
For $n\ge1$, write
$
\overline v^{\,n}:=v^n-v^*.
$

\begin{Lemma}\label{lm:PIA.est1}
	Under Assumption \ref{assum:weight.discounting}, there exists a finite  constant $C>0$, independent of $n$, $\lambda$, $(t,s)\in \Delta$, and $(i,j)\in \Sc^2$, such that
	\begin{align}
		\|\pi^n(t,i)-\pi^*(t,i)\|_{\mathrm{TV}} &\leq \frac{Cw(i)}{\lambda}\|\overline v^{\,n-1}\|_{\widetilde \Xc},  \label{eq:lm.est01}\\
		\left| g^{\pi^n(s,j)}(j;t,i) - g^{\pi^*(s,j)}(j;t,i) \right|&\leq \frac{Cw(j)}{\lambda} \|\overline v^{\,n-1}\|_{\widetilde \Xc}, \label{eq:lm.est02}\\
		\lambda \left| \Hc(\pi^n(t,i))-\Hc(\pi^*(t,i)) \right| &\leq \frac{Cw(i)}{\lambda} \|\overline v^{\,n-1}\|_{\widetilde \Xc}. \label{eq:lm.est03}
	\end{align}
\end{Lemma}

\begin{proof}
	Under \eqref{eq:PIA.zero.drift}, the Hamiltonian $H$ defined in \eqref{eq:def.H} becomes $H(t,i,, y, a)  := g(i,a;t,i)+p(t,i,a)\cdot y.$
	Fix $(t,i)\in\N_0\times\Sc$. By \eqref{eq:PIA.zero.drift} in Assumption \ref{assum:weight.discounting},
	\begin{align}
		&\|H(t, i, v^{n-1}(\cdot; t,i), a) - H(t, i, v^*(\cdot; t,i), a) \|_{L^\infty(A)} \notag\\
		&\leq b_{W} \|v^{n-1}(\cdot;t,i)-v^*(\cdot;t,i)\|_{\wt} 
		\leq b_{W}w(i) \|\overline v^{\,n-1}\|_{\widetilde \Xc}.\label{eq:lm.pia0}
	\end{align}
	Then applying Lemma \ref{lm:PIA.est0} gives \eqref{eq:lm.est01}.
	
	Since adding a constant to the integrand does not change the integral against the signed measure $\pi^n(s,j)-\pi^*(s,j)$, we may choose $c_{t,i,j}\in\R$ such that
	$$
	\|g(j,\cdot;t,i)-c_{t,i,j}\|_{L^\infty(A)}  =  \frac12\operatorname{osc}_{a\in A}g(j,a;t,i)  \leq\frac{K_g}{2},
	$$
	which gives
	$$
	\begin{aligned}
		\left| g^{\pi^n(s,j)}(j;t,i) - g^{\pi^*(s,j)}(j;t,i) \right| \leq 2\|g(j,\cdot;t,i)-c_{t,i, j}\|_{L^\infty(A)} \|\pi^n(s,j)-\pi^*(s,j)\|_{\mathrm{TV}},
	\end{aligned}
	$$
	and \eqref{eq:lm.est02} follows from \eqref{eq:lm.est01}.
	
	It remains to prove \eqref{eq:lm.est03}. For $s\in[0,1]$, set
	$$
	H_s :=sH(t, i, v^{n-1}(\cdot; t,i), a)+(1-s)H(t, i, v^{*}(\cdot; t,i), a) , \quad
	\pi_s(a) :=  \frac{\exp(H_s(a)/\lambda)}{\int_A\exp(H_s(a')/\lambda)\,da'},
	$$
	and let $\Delta H(a):=H\big(t,i,v^{n-1}(\cdot;t,i),a\big)-H\big(t,i,v^*(\cdot;t,i),a\big).$
	Then
	\be\label{eq:lm.estpiderive0} 
	\Hc(\pi_s) = \log\int_Ae^{H_s(a)/\lambda}\,da  - \frac1\lambda\int_AH_s(a)\pi_s(a)\,da.
	\ee 
	Direct calculations give
	$$
	\begin{cases}
		\frac{d}{ds}\log \left(  \int_A e^{H_s(a)/\lambda}\,da \right)=\frac1\lambda\int_A\Delta H(a)\pi_s(a)\,da,\\
		\partial_s\pi_s(a)=\frac1\lambda\pi_s(a)\left[\Delta H(a)-\int_A\Delta H(a')\pi_s(a')\,da'\right].
	\end{cases}
	$$
	Differentiating \eqref{eq:lm.estpiderive0} then combining with above formulas gives
	\be\label{eq:lm.holder.7}
	\begin{aligned}
		\frac{d}{ds}\Hc(\pi_s) &=-\frac1\lambda \int_AH_s(a)\partial_s\pi_s(a)\,da\\
		&= -\frac1{\lambda^2} \operatorname{Cov}_{\pi_s} \bigg( H_s(a), H\big(t,i,v^{n-1}(\cdot;t,i),a\big) - H\big(t,i,v^*(\cdot;t,i),a\big) \bigg),
	\end{aligned}
	\ee
	where, any probability density $q$ and bounded functions $\xi,\eta$,
		\be\label{eq:lm.holder.5}
		|\operatorname{Cov}_{q}(\xi,\eta)|:=\int_A\left(\xi(a)-\int_A\xi(a')q(a')\,da'\right)\eta(a)q(a)\,da \leq\operatorname{osc}_A(\xi) \|\eta\|_{L^\infty(A)}.
		\ee
		Notice that \eqref{eq:PIA.Hamiltonian.osc} ensures
		$
		\operatorname{osc}_A(H_s)\leq K_{H}.
		$
		Combining this with \eqref{eq:lm.pia0}, \eqref{eq:lm.holder.7} and \eqref{eq:lm.holder.5} gives
		$$
		\begin{aligned}
			\left|  \frac{d}{ds}\Hc(\pi_s) \right| \leq & \frac{\operatorname{osc}_A(H_s)}{\lambda^2} \|H(t, i, v^{n-1}(\cdot; t,i), a) -H(t, i, v^{*}(\cdot; t,i), a) \|_{L^\infty(A)} \\
			\leq & \frac{b_wK_{H}(1+w(i))}{\lambda^2}\|\overline v^{\,n-1}\|_{\widetilde \Xc}.
		\end{aligned}
		$$
		Then integrating over $s\in[0,1]$ and multiplying by $\lambda$ proves
		\eqref{eq:lm.est03}.
	\end{proof}
	
	\begin{Lemma}\label{lm:PIA.rho.uniform}
		There exists a finite constant $C_\lambda>0$, independent of
		$\rho\in[0,\rho_0]$, $\pi=\pi^n$, and $n$, such that
		$$
		\|U_\rho^{\pi^n}\|_{\widehat \Xc}
		+
		\|U_\rho^{\pi^*}\|_{\widehat \Xc}
		\leq C_\lambda.
		$$
	\end{Lemma}
	
	\begin{proof}
		During the proof, let $C_\lambda>0$ being a finite constant dependent of $\lambda$ that may vary from line to line. The reward-growth condition in Assumption \ref{assum:reward-bound} and the entropy estimate of Lemma \ref{lm:entropy.est}, together with $v^n,v^*\in\Kc(M_\lambda^*)$, yield
		\be\label{eq:lm.piaestU} 
		|g^{\pi(s,j)}(j;t,i)| +  \lambda|\Hc(\pi(s,j))| \leq C_\lambda\big(w(i)+w(j)\big),\; \text{  for $\pi=\pi^n$ or $\pi^*$}.
		\ee 
		By \eqref{eq:PIA.zero.drift},  $\E[w(X^\pi_{s+m}) \mid X^\pi_s=j]\leq b_w$, for all $m\geq1$ and $\pi=\pi^n$ or $\pi^*$. Consequently,$$ 
		|U_\rho^\pi(s,j;t,i)|  \leq C_\lambda  \sum_{m=0}^{\infty}\rho^{m+1}  \big( w(i)+w(j)+b_w \big).
		$$
		This together with $0\leq \rho\le\rho_0<1$ yields
		$$  |U_\rho^\pi(s,j;t,i)|  \leq  \frac{C_\lambda\rho_0}{1-\rho_0}  \big(w(i)+w(j)\big).$$
		Then dividing the above inequality by $W(j)w(i)$, taking the supremum, and using $\sup_jw(j)/W(j)<\infty$ leads to the desired result.
	\end{proof}
	
	\begin{Remark}\label{rm:Clambda}
		The analysis for \eqref{eq:lm.piaestU} implies that the constant $C_\lambda$ relates to $\lambda$ only through $M^*_\lambda$. Therefore, by Lemma \ref{lm:self-mapping}, $C_\lambda$ can be chosen independently over $\lambda\in(0,1]$.
	\end{Remark}

\begin{proof}[Proof of Theorem \ref{thm:converge.PI}]
	Through the proof, let $C>0$ being a finite constant independent of $n, \rho$ that can be chosen uniformly over $\lambda\in(0,1]$ and may vary from line to line.  For $\rho\in[0,\rho_0]$, set
	$$
	U_\rho^n:=U_\rho^{\pi^n}, \quad U_\rho^*:=U_\rho^{\pi^*}, \quad \overline U_\rho^{\,n}:=U_\rho^n-U_\rho^*.
	$$
	Subtracting the two identities in \eqref{eq:vn.rho} with $\pi=\pi^n, \pi^*$ gives
	\begin{align}
		\overline U_\rho^{\,n}(s,j;t,i)=
		\rho\Bigg[ &g^{\pi^n(s,j)}(j;t,i)-g^{\pi^*(s,j)}(j;t,i)+\lambda\big( \Hc(\pi^n(s,j))-\Hc(\pi^*(s,j)) \big)
		\nonumber\\
		&+ \sum_{r\in\Sc} \big( p^{\pi^n}(s,j,r)-p^{\pi^*}(s,j,r) \big) U_\rho^n(s+1,r;t,i) \nonumber\\
		&+ \sum_{r\in\Sc} p^{\pi^*}(s,j,r) \overline U_\rho^{\,n}(s+1,r;t,i) \Bigg].
		\label{eq:thm.PIA0}
	\end{align}
	We now estimate the four terms on the right-hand side. By \eqref{eq:lm.est02}--\eqref{eq:lm.est03},
	\be\label{eq:thm.PIA1}
	\begin{aligned}
		&\left|g^{\pi^n(s,j)}(j;t,i)   -g^{\pi^*(s,j)}(j;t,i)  \right|  +   \lambda   \left|   \Hc(\pi^n(s,j))-\Hc(\pi^*(s,j))   \right|
		\leq C\lambda^{-1} w(j)  \|\overline v^{\,n-1}\|_{\widetilde \Xc}.
	\end{aligned}
\end{equation}
For the transition-difference term, the definition of total variation,
\eqref{eq:PIA.zero.drift}, Lemma~\ref{lm:PIA.rho.uniform}, Remark \ref{rm:Clambda} and
\eqref{eq:lm.est01} imply, for all $(t,s)\in \Delta$ and $(i,j)\in \Sc^2$ that,
\begin{align}
	& \left| \sum_{r\in\Sc} \big( p^{\pi^n(s,j)}(s,j,r)-p^{\pi^*(s,j)}(s,j,r) \big) U_\rho^n(s+1,r;t,i) \right| \nonumber\\
	&\leq Cb_{W} \|\pi^n(s,j)-\pi^*(s,j)\|_{\mathrm{TV}}
	\|U_\rho^n(s+1,\cdot;t,i)\|_{\wt}
	\leq \frac{Cw(j)w(i)}{\lambda} \|\overline v^{\,n-1}\|_{\widetilde \Xc}.
	\label{eq:thm.PIA2}
\end{align}
Similarly,
\be\label{eq:thm.PIA2b}
\begin{aligned}
	\left| \sum_{r\in\Sc}p^{\pi^*(s,j)}(s,j,r) \overline U_\rho^{\,n}(s+1,r;t,i) \right| \leq &  \sum_{r\in\Sc}p^{\pi^*(s,j)}(s,j,r) W(r) \| \overline U_\rho^{\,n}(s+1,\cdot ;t,i) \|_{\wt}\\
	\leq & b_{W}w(i) \|\overline U_\rho^{\,n}\|_{\widehat \Xc},\quad \forall (t,s)\in \Delta, \forall (i,j)\in \Sc^2.
\end{aligned}
\ee
Notice that 
$
\sup_{j\in\Sc}\frac{1+w(j)}{W(j)}<\infty.
$
Dividing both side of \eqref{eq:thm.PIA0} by $W(j)w(i)$ then plugging  the estimates \eqref{eq:thm.PIA1}--\eqref{eq:thm.PIA2b} into \eqref{eq:thm.PIA0} yields
$$
\|\overline U_\rho^{\,n}\|_{\widehat \Xc} \leq
\rho\left[C_1\lambda^{-1}\|\overline v^{\,n-1}\|_{\widetilde \Xc} + C_0\|\overline U_\rho^{\,n}\|_{\widehat \Xc} \right],
$$
where $C_0, C_1$ are finite constants only depends on $C, b_W$ and $ \sup_{j\in\Sc}\frac{w(j)}{W(j)}$.
Hence, whenever $\rho_0C_0<1$,
\be\label{eq:thm.PIA3}
\|\overline U_\rho^{\,n}\|_{\widehat \Xc} \leq \frac{ \rho C_1 \lambda^{-1} }{  1-\rho C_0 }\|\overline v^{\,n-1}\|_{\widetilde \Xc}
\leq  \frac{ \rho_0 C_1 \lambda^{-1} }{  1-\rho_0 C_0 } \|\overline v^{\,n-1}\|_{\widetilde \Xc},\quad \forall \rho\in(0,\rho_0].
\ee  
Notice from \eqref{eq:PIA.mixture.identity} that
$$
v^n(j;t,i)-v^*(j;t,i) = \int_{[0,\rho_0]}  \overline U_\rho^{\,n}(t+1,j;t,i)\,F(d\rho),
$$
which together with \eqref{eq:thm.PIA3} gives
$$
\|\overline v^{\,n}\|_{\widetilde \Xc}  \leq \int_{[0,\rho_0]}  \|\overline U_\rho^{\,n}\|_{\widehat \Xc}\,F(d\rho)
\leq \theta_\lambda(\rho_0)  \|\overline v^{\,n-1}\|_{\widetilde \Xc}. 
$$
Then iterating the above inequality leads to the desired result.
\end{proof}

\begin{proof}[Proof of Corollary \ref{cor:PIA.epsilon}]
	Throughout the proof, $C>0$ denotes a finite constant which is independent of $n$, $\lambda\in(0,1]$, $(t,i)\in\N_0\times\Sc$, and which may vary from line to line. We use the condition \eqref{eq:cor.osc} in addition to the assumptions of Theorem~\ref{thm:converge.PI}. We divide the proof into two steps.
	
	\medskip
	\noindent
	\textbf{Step 1. }
	We first establish the estimates for the regularized and unregularized Hamiltonians that will be used below. Define
	\begin{align}
		\widetilde H(t,i,y)
		&:= \max_{a\in A} H(t,i,y,a) = \max_{a\in A} \left[  g(i,a;t,i)+p(t,i,a)\cdot y  \right],
		\label{eq:PIA.H0}
		\\
		\widetilde H_\lambda(t,i,y)&:=  \max_{\varpi\in \Dc(A)} \left\{ \int_A H(t,i,y,a)\varpi(a)da+\lambda \Hc(\varpi) \right\} \notag\\
		&=\lambda\log  \int_A  \exp\left( \frac{ g(i,a;t,i)+p(t,i,a)\cdot y  }{\lambda} \right)\,da.
		\label{eq:PIA.Hlambda}
	\end{align}
	Fix an arbitrary $(t,i)\in \N_0\times \Sc$ and $y\in\Kc(M^*)$.  We first conclude from \eqref{3.10} that
	\be\label{eq:coreps.step1.0}
	\begin{aligned}
		\widetilde  H_\lambda(t,i,y)-\widetilde H(t,i,y)&\leq \widetilde H_\lambda-\int_A H(t, i, y, a)\Gamma_\lambda(t,y, i)(a)da\\
		& \leq \lambda\Hc(\Gamma_\lambda(t, y, i))\leq \lambda\log\Leb(A).
	\end{aligned}
	\ee
	Then choose a maximizer $a^*\in A$. By
	\eqref{eq:cor.osc},
	$$
	g(i,a;t,i)+p(t,i,a)\cdot y
	\ge
	\widetilde H(t,i,y)-L_{H}|a-a^*|.
	$$
	By Assumption \ref{assum:cone}, there exists a finite constant $c_A>0$ independent of $a^*, t, i, y$ such that
	$$
	\Leb(A\cap B_r(a^*))\geq c_Ar^\ell,\quad \forall r\in(0,\vartheta].
	$$
	Then, for any $r\in(0,\vartheta]$, $\int_A \exp(H(t, i, y, a)/\lambda) da\geq \int_{A\cap B_r(a^*)} \exp[(\widetilde H(t,i,y)-L_{H}r )/\lambda] da$. Consequently,
	$$
	\begin{aligned}
		\widetilde H_\lambda(t,i,y)
		&\geq \lambda\log \int_{A\cap B_r(a^*)} \exp\left( \frac{  \widetilde H(t,i,y)-L_{H}r }{\lambda} \right)\,da \\
		&\geq\widetilde H(t,i,y) -L_{H}r + \lambda\log(c_Ar^\ell),\quad \forall r\in(0,\vartheta].
	\end{aligned}
	$$
	Taking $r=\lambda\wedge \vartheta$ gives
	$$
	\widetilde H(t,i,y)-\widetilde H_\lambda(t,i,y)
	\leq
	C\lambda(1+|\log\lambda|).
	$$
	By combining the above estimate with \eqref{eq:coreps.step1.0}, we obtain
	\be\label{eq:coreps.step1.1}
	\left| \widetilde H_\lambda(t,i,y) - \widetilde H(t,i,y)
	\right| \leq C\lambda(1+|\log\lambda|), \quad \forall (t,i)\in\N_0\times\Sc,\quad y\in\Kc(M^*).
	\ee
	Notice that the constant $C$ is uniform in $(t,i)$ and
	$y\in\Kc(M^*)$ precisely because the Lipschitz constant in
	\eqref{eq:cor.osc} and the cone constants in
	Assumption~\ref{assum:cone} are uniform. 
	
	We now derive Lipschitz continuities of
	$\widetilde H$ and $\widetilde H_\lambda$. Let $y_1,y_2\in\Bc_{\wt}$. Then
	$$
	\begin{aligned}
		\widetilde H(t,i,y_1)-\widetilde H(t,i,y_2)&\leq\sup_{a\in A} p(t,i,a)\cdot(y_1-y_2)
		\leq\sup_{a\in A}\sum_{j\in\Sc} p(t,i,j,a)W(j) \frac{|y_1(j)-y_2(j)|}{W(j)}.
	\end{aligned}
	$$
	Combining thi with \eqref{eq:PIA.zero.drift} gives
	$
	\widetilde H(t,i,y_1) -\widetilde H(t,i,y_2) \leq b_W\|y_1-y_2\|_{\wt}.
	$
	Then by interchanging $y_1$ and $y_2$, we have
	\be\label{eq:coreps.step1.2}
	\left| \widetilde H(t,i,y_1) -\widetilde H(t,i,y_2) \right| \leq b_W\|y_1-y_2\|_{\wt}.
	\ee
	Take $F_1, F_2\in L_\infty(A)$. By $F_1(a)\leq F_2(a)+\|F_1-F_2\|_{L^\infty(A)}$, 
	$$\int_A \exp(F_1(a)/\lambda) da\leq e^{\frac{\|F_1-F_2\|_{L^\infty(A)}}{\lambda}}\int_A \exp(F_2(a)/\lambda) da$$
	which leads to $
	\lambda\log\int_Ae^{F_1(a)/\lambda}\,da-\lambda\log\int_Ae^{F_2(a)/\lambda}\,da\leq \|F_1-F_2\|_{L^\infty(A)}$. And interchanging $F_1$ and $F_2$ gives
	$$
	\left| \lambda\log\int_Ae^{F_1(a)/\lambda}\,da - \lambda\log\int_Ae^{F_2(a)/\lambda}\,da
	\right|
	\leq
	\|F_1-F_2\|_{L^\infty(A)},
	$$
	which together with \eqref{eq:PIA.zero.drift}, yields
	\be\label{eq:coreps.step1.3}
	\left|
	\widetilde H_\lambda(t,i,y_1)
	-
	\widetilde H_\lambda(t,i,y_2)
	\right|
	\leq
	b_W\|y_1-y_2\|_{\wt}.
	\ee
	
	We next estimate the difference between the regularized and  unregularized continuation values under $\pi^n$. Since $ \pi^n=\Gamma_\lambda(v^{n-1}) $
	and, by Lemma~\ref{lm:self-mapping}, $v^{n-1}\in\Kc(M^*),$ Lemma~\ref{lm:entropy.est} implies
	$$
	\begin{aligned}
		|\Hc(\pi^n(s,r))| &\leq  C_0+C_1|\ln\lambda|  +C_2\ln(1+M^*) +C_3w(r)
		\leq  C(1+|\ln\lambda|)w(r),\; \forall (s,r)\in \N_0\times \Sc,
	\end{aligned}
	$$
	where we have used $w(r)\geq1$ and the uniform boundedness of
	$M^*$ over $\lambda\in(0,1]$ from
	Lemma~\ref{lm:self-mapping}. Hence, for every
	$(t,i,j)\in\N_0\times\Sc^2$,
	$$
	\begin{aligned}
		&  \left|  V_\lambda^{\pi^n}(j;t,i)  -  V^{\pi^n}(j;t,i)  \right|\leq
		C\lambda(1+|\ln\lambda|)  \E\left[  \sum_{m=1}^\infty  \widetilde\delta(m)  w(X_{t+m}^{\pi^n})  \,\Bigm|\,  X_{t+1}^{\pi^n}=j \right].
	\end{aligned}
	$$
	For $m=1$, $ \E\left[w(X_{t+1}^{\pi^n}) \,\Bigm|\,X_{t+1}^{\pi^n}=j \right] = w(j),$
	while, an induction with \eqref{eq:PIA.zero.drift} gives
	$$
	\E\left[ w(X_{t+m}^{\pi^n}) \,\Bigm|\, X_{t+1}^{\pi^n}=j \right] \leq b_w,\quad \forall m\geq2.
	$$
	Therefore,
	$$
	\begin{aligned}
		&  \left|  V_\lambda^{\pi^n}(j;t,i)  -  V^{\pi^n}(j;t,i)  \right|\leq  C\lambda(1+|\ln\lambda|)  \left[  \widetilde\delta(1)w(j)  +  b_w\sum_{m=2}^{\infty}\widetilde\delta(m)
		\right].
	\end{aligned}
	$$
	Recall \eqref{eq: assum.wW0} and notice that
	$ \sum_{m=1}^{\infty}\widetilde\delta(m) = \int_{[0,\rho_0]} \frac{\rho}{1-\rho}\,F(d\rho)
	\leq\frac{\rho_0}{1-\rho_0} <\infty.$ We then conclude that
	\be\label{eq:PIA.Vlambda.V.bias.new}
	\left\|
	V_\lambda^{\pi^n}(\cdot;t,i)
	-
	V^{\pi^n}(\cdot;t,i)
	\right\|_{\wt}
	\leq
	C\lambda(1+|\ln\lambda|),
	\qquad
	\forall(t,i)\in\N_0\times\Sc.
	\ee 
	Also, applying Lemma~\ref{lm:entropy.est} for $\pi^n(t,i)=\Gamma_\lambda(t, v^{n-1}(t,i), i)$ gives
	\be\label{eq:PIA.current.entropy.new}
	\lambda|\Hc(\pi^n(t,i))| \leq C\lambda(1+|\ln\lambda|)w(i).
	\ee
	
	\medskip
	\noindent
	\textbf{Step 2.} Now we provide regret estimate for $\pi^n$. Fix $(t,i)\in\N_0\times\Sc$ and $n\geq1$.
	Notice that
	\be\label{eq:PIA.regret.start}
	\begin{cases}
		\sup_{\varpi\in\Pc(A)} J^{\varpi\otimes_1\pi^n}(t,i)= \widetilde H(t, i, V^{\pi^n}(\cdot; t,i)),\\
		J^{\pi^n}(t,i) = g^{\pi^n(t,i)}(i;t,i) + p^{\pi^n(t,i)}(t,i)\cdot V^{\pi^n}(\cdot; t, i).
	\end{cases}
	\ee
	By 
	$
	\pi^n(t,i)=\Gamma_\lambda(v^{n-1})(t,i),
	$
	\be\label{eq:PIA.Gibbs.identity.new}
	\begin{aligned}
		\widetilde H_\lambda(t,i, v^{n-1}(\cdot;t,i))  =
		&\, g^{\pi^n(t,i)}(i;t,i)  + p^{\pi^n(t,i)}(t,i)\cdot  v^{n-1}(\cdot;t,i)+  \lambda\Hc(\pi^n(t,i)).
	\end{aligned}
	\ee
	By combining \eqref{eq:PIA.Gibbs.identity.new} with
	\eqref{eq:PIA.regret.start}, we obtain
	$$
	\begin{aligned}
		\sup_{\varpi\in\Pc(A)} J^{\varpi\otimes_1\pi^n}(t,i) - J^{\pi^n}(t,i) ={}&  \widetilde H(t,i,V^{\pi^n}(\cdot;t,i))  -  \widetilde H_\lambda(t,i, v^{n-1}(\cdot;t,i))  \\
		&+  \lambda\Hc(\pi^n(t,i))  + p^{\pi^n(t,i)}(t,i)\cdot  ( v^{n-1}(\cdot;t,i)-V^{\pi^n}(\cdot;t,i))\\
		={}& \Big[  \widetilde H(t,i, v^{n-1}(\cdot;t,i)) -  \widetilde H_\lambda(t,i, v^{n-1}(\cdot;t,i))  \Big]\\
		&+ \Big[  \widetilde H(t,i,V^{\pi^n}(\cdot;t,i))-  \widetilde H(t,i, v^{n-1}(\cdot;t,i))  \Big]   \\
		&+  \lambda\Hc(\pi^n(t,i))+  p^{\pi^n(t,i)}(t,i)\cdot ( v^{n-1}(\cdot;t,i)-V^{\pi^n}(\cdot;t,i)).
	\end{aligned}
	$$
	By applying \eqref{eq:coreps.step1.1}, \eqref{eq:coreps.step1.2}, \eqref{eq:PIA.current.entropy.new} and \eqref{eq:PIA.zero.drift} into above inequality, we get
	\be\label{eq:PIA.regret.mid} 
	\begin{aligned}
		&\sup_{\varpi\in\Pc(A)} J^{\varpi\otimes_1\pi^n}(t,i) - J^{\pi^n}(t,i)\\
		& \leq  C\lambda(1+|\ln\lambda|)+  b_W\|V^{\pi^n}(\cdot;t,i)- v^{n-1}(\cdot;t,i)\|_{\wt} \\
		&\quad + C\lambda(1+|\ln\lambda|)w(i)  +  b_W\|V^{\pi^n}(\cdot;t,i)- v^{n-1}(\cdot;t,i)\|_{\wt}\\
		&\leq  2b_W\|V^{\pi^n}(\cdot;t,i)- v^{n-1}(\cdot;t,i)\|_{\wt} + C\lambda(1+|\ln\lambda|)w(i).
	\end{aligned}
	\ee 
	
	We now estimate the first term in the last line of \eqref{eq:PIA.regret.mid}. Since $
	v^n=V_\lambda^{\pi^n}$, 
	\be\label{eq:corest.step2.1} 
	\begin{aligned}
		&  \|V^{\pi^n}(\cdot;t,i)- v^{n-1}(\cdot;t,i)\|_{\wt}
		=  \left\| V^{\pi^n}(\cdot;t,i) -  v^{n-1}(\cdot;t,i)  \right\|_{\wt}  \\
		&\leq \left\|  V^{\pi^n}(\cdot;t,i)  -  V_\lambda^{\pi^n}(\cdot;t,i)  \right\|_{\wt}
		+  \left\|  v^n(\cdot;t,i)  -  v^{n-1}(\cdot;t,i)  \right\|_{\wt}.
	\end{aligned}
	\ee
	Since 
	$$
	\left\|  v^n(\cdot;t,i)  -  v^{n-1}(\cdot;t,i)  \right\|_{\wt}\leq
	w(i)  \|v^n-v^{n-1}\|_{\widetilde\Xc},
	$$
	Theorem~\ref{thm:converge.PI} together with  \eqref{eq:PIA.exponential} tells
	$$
	\begin{aligned}
		\|v^n-v^{n-1}\|_{\widetilde\Xc}
		&\leq \|v^n-v^*\|_{\widetilde\Xc}  +  \|v^{n-1}-v^*\|_{\widetilde\Xc}  \\
		&\leq  \left[ \theta_\lambda(\rho_0)^n  +  \theta_\lambda(\rho_0)^{n-1}  \right]\|v^0-v^*\|_{\widetilde\Xc}  
		\leq  2\theta_\lambda(\rho_0)^{n-1}  \|v^0-v^*\|_{\widetilde\Xc},
	\end{aligned}
	$$
	where we used $\theta_\lambda(\rho_0)<1$ from \eqref{eq:PIA.small.rho}. By plugging this and \eqref{eq:PIA.Vlambda.V.bias.new} into \eqref{eq:corest.step2.1}, we  reach to
	\be\label{eq:PIA.V.vnminus1.new}
	\|V^{\pi^n}(\cdot;t,i)- v^{n-1}(\cdot;t,i)\|_{\wt}
	\leq  C\lambda(1+|\ln\lambda|)+  2w(i)  \theta_\lambda(\rho_0)^{n-1}  \|v^0-v^*\|_{\widetilde\Xc}.
	\ee
	Substituting \eqref{eq:PIA.V.vnminus1.new} into \eqref{eq:PIA.regret.mid} and using again $w(i)\geq1$ yields
	$$
	\begin{aligned}
		\sup_{\varpi\in\Pc(A)} J^{\varpi\otimes_1\pi^n}(t,i) - J^{\pi^n}(t,i)  \leq & C w(i)  \Big[ \lambda(1+|\ln\lambda|)+  \theta_\lambda(\rho_0)^{n-1}  \|v^0-v^*\|_{\widetilde\Xc}  \Big].
	\end{aligned}
	$$
	Then by the arbitrariness of $(t,i)$ is arbitrary, we conclude that $\pi^n$ is a weighted $\eps_n$-equilibrium, where by Theorem \ref{thm:converge.PI},
	$$
	\eps_n \leq C\Big[ \lambda(1+|\ln\lambda|)+  \theta_\lambda(\rho_0)^{n-1}  \|v^0-v^*\|_{\widetilde\Xc}  \Big]=C\left[\lambda(1+|\ln\lambda|) +  \left( \frac{  \rho_0C_1\lambda^{-1}}{ 1-\rho_0C_0} \right)^{n-1} \|v^0-v^*\|_{\widetilde\Xc}\right].
	$$
\end{proof}

\noindent{\bf Acknowledgements:}
Zhenhua Wang is supported by Shandong Excellent Young Scientists Fund Program (Overseas) under grant No. 2025HWYQ–022 and National Natural Science Foundation of China (NSFC) under grant no.12501659. 

\bibliographystyle{plain}
\bibliography{reference}

\end{document}